\documentclass[english]{article}
\usepackage[T1]{fontenc}
\usepackage[utf8]{inputenc}
\usepackage{babel}
\usepackage{bbold}
\usepackage{mathtools}
\usepackage{amsthm}
\usepackage{stmaryrd}
\usepackage{amsmath}
\usepackage{enumitem}
\usepackage{pifont}
\RequirePackage[hyperref,backend=bibtex, 
style=numeric,maxnames=99
]{biblatex}
\usepackage{multicol}
\usepackage[mathscr]{euscript}
\usepackage{amssymb}
\usepackage{multicol}
\usepackage[a4paper,left=3.5cm,right=3.5cm,top=3cm,bottom=3cm]{geometry}
\usepackage{tcolorbox}
\usepackage{dsfont}
\usepackage{fancyheadings}
\newtheorem{theorem}{Theorem}[section]
\newtheorem{crl}[theorem]{Corollary}
\newtheorem{prop}[theorem]{Proposition}
\newtheorem{lm}[theorem]{Lemma}
\newtheorem{defi}[theorem]{Definition}
\newtheorem{rmq}[theorem]{Remark}
\newtheorem{ex}[theorem]{Example}
\newtheorem{exs}[theorem]{Examples}
\newtheorem*{nota}{Notation}
\newtheorem*{notas}{Notations}
\def\restriction#1#2{\mathchoice
	{\setbox1\hbox{${\displaystyle #1}_{\scriptstyle #2}$}
		\restrictionaux{#1}{#2}}
	{\setbox1\hbox{${\textstyle #1}_{\scriptstyle #2}$}
		\restrictionaux{#1}{#2}}
	{\setbox1\hbox{${\scriptstyle #1}_{\scriptscriptstyle #2}$}
		\restrictionaux{#1}{#2}}
	{\setbox1\hbox{${\scriptscriptstyle #1}_{\scriptscriptstyle #2}$}
		\restrictionaux{#1}{#2}}}
\def\restrictionaux#1#2{{#1\,\smash{\vrule height .8\ht1 depth .85\dp1}}_{\,#2}}

\newcommand{\rr}{\mathbb{R}}

\newcommand{\spn}{\text{Span}}
\newcommand{\nn}{\mathbb{N}}
\newcommand{\zz}{\mathbb{Z}}
\newcommand{\dt}{\mathrm{d}t}

\newcommand{\ds}{\mathrm{d}s}
\newcommand{\ph}{\varphi}
\newcommand{\ep}{\varepsilon}
\newcommand{\B}{\mathcal{B}}
\newcommand{\ev}{\textsc{e}}
\newcommand{\dd}{\,\mathrm{d}}
\DeclareMathOperator{\Br}{Br}
\DeclareMathOperator{\ad}{ad}
\numberwithin{equation}{section}
\title{Quadratic obstructions to Small-Time Local Controllability for multi-input systems}
\author{Théo Gherdaoui\thanks{Univ Rennes, CNRS, IRMAR - UMR 6625, F-35000 Rennes, France}}
\usepackage[pdfborder={0 0 0}]{hyperref}
	
\begin{document}
	\maketitle
	\begin{abstract}
		We present a necessary condition for the small-time local controllability of multi-input control-affine systems on $\rr^d$. This condition is formulated on the vectors of $\rr^d$  resulting from the evaluation at zero of the Lie brackets of the vector fields: it involves both their direction and their amplitude.
		
		The proof is an adaptation to the multi-input case of a general method introduced by Beauchard and Marbach in the single-input case -- see \cite{beauchard2024unified}. It relies on a Magnus-type representation formula: the state is approximated by a linear combination of the evaluation at zero of the Lie brackets of the vector fields, whose coefficients are functionals of the time and the controls. Finally, obstructions to small-time local controllability result from interpolation inequalities.\par
		\vspace{0.25 cm}
	\end{abstract}
	\tableofcontents
	\section{Introduction}
	\subsection{Definitions}
	We consider the control-affine system
	\begin{equation}\label{affine-syst}\left\lbrace\begin{array}{rcl}x'(t)&=&f_0(x(t))+u(t)f_1(x(t))+v(t)f_2(x(t))\\x(0)&=&0\end{array}\right.,
	\end{equation}where the state is $x(t)\in\rr^d$, the controls are scalar functions $u(t),v(t)\in\rr$ and $f_0,f_1$ and $f_2$ are real-analytic vector fields on a neighborhood of $0$ in $\rr^d$ with $f_0(0)=0$. The last hypothesis ensures that $0$ is an equilibrium of the free system (\textit{i.e.}\ with $(u,v)\equiv 0)$. 
	
	For each $t>0$, $u,v\in L^1((0,t),\rr)$, there exists a unique maximal mild solution to \eqref{affine-syst}, which we will denote by $x(\cdot;(u,v))$. As we are interested in small time and small controls, the solution is well-defined up to time $t$.
	
	In this article, we study the small-time local controllability of system \eqref{affine-syst} in the sense of Definition \ref{stlc} below, which requires the following notion. For $t > 0$, $m\in\nn$ and $p\in[1,+\infty]$, we consider the usual Sobolev space $W^{m,p}((0,t),\rr)$ equipped with the
	usual norm $$\left\|u\right\|_{W^{m,p}}:=\left\|u\right\|_{L^p}+\cdots+\|u^{(m)}\|_{L^p}.$$ 
	The following concept was introduced by Beauchard and Marbach in \cite{beauchard2017quadratic} for scalar-input systems.
	\begin{defi}[$W^{m,p}\times W^{m',p'}$-STLC]
		\label{stlc} Let $m,m'\in\nn$, $p,p'\in[1,+\infty]$. We say that system \eqref{affine-syst} is $W^{m,p}\times W^{m',p'}$-STLC when, for every $t,\rho>0$, there exists $\delta>0$, such that, for every $x^*\in B(0,\delta)$, there exist $(u,v)\in W^{m,p}((0,t),\rr)\times W^{m',p'}((0,t),\rr)$ with $\left\|(u,v)\right\|_{W^{m,p}\times W^{m',p'}}\leq \rho$ and $x(t;(u,v))=x^*$.
		
		\noindent \\We say that system \eqref{affine-syst} is $W^{m,p}$-STLC when \eqref{affine-syst} is $W^{m,p}\times W^{m,p}$-STLC.
	\end{defi}
	\begin{rmq}
		The historical notion of STLC corresponds to $m=0$, $p=\infty$ -- see \cite{coronbook,sussmann,stefani}.
	\end{rmq}
	We are looking for necessary conditions for STLC formulated in terms of Lie brackets; the definition is recalled below. Indeed, Krener proved in \cite[Theorem $1$]{krener} that all the information concerning the STLC is contained in the evaluation at 0 of the Lie brackets of the vector fields $f_0$, $f_1$ and $f_2$.
\begin{defi}[Lie bracket of vector fields]\label{def:lie-bracket-champ}
		Let $f,g:\Omega\to\rr^d$ be two smooth (at least $\mathcal{C}^1$) vector fields on an open subset $\Omega$ of $\rr^d$. We define
		\begin{equation}
			[f,g]:x\in\Omega\mapsto \mathrm{D}g(x)\cdot f(x) - \mathrm{D}f(x)\cdot g(x).
		\end{equation}
	\end{defi}
	\subsection{Illustrative examples}
	We propose to study the following example of a quadratic system in an elementary way, as it is representative of the type considered in this article.
	\begin{ex}\label{firstexample} Let us consider 
		\begin{equation}\label{jouet}\left\lbrace\begin{array}{rcl}
				x_1'&=&u\\x_2'&=&v\\x_3'&=&x_1^2+x_2^2+\alpha x_1x_2\end{array}\right.,
		\end{equation}
			with $\alpha\in\rr$. This is a control-affine system of the form \eqref{affine-syst} with $$f_0(x)=\begin{pmatrix}0\\0\\x_1^2+x_2^2+\alpha x_1x_2\end{pmatrix},\quad f_1(x)=e_1,\quad f_2(x)=e_2.$$
		For all $x,y\in\rr$, 
		$$x^2+y^2+\alpha xy=\left(x+\frac{\alpha}{2}y\right)^2+\frac{4-\alpha^2}{4}y^2.$$ Consequently, if $|\alpha|\leq 2$, then $x_3'\geq 0$ and \eqref{jouet} is not $L^{\infty}$-STLC.
		Note that, when $|\alpha|>2$, for every $m\in\nn$,  the system is $W^{m,\infty}-STLC$. Indeed, we can choose $u:s\in(0,t)\mapsto\varphi'(\frac{s}{t})$ and $v:s\in(0,t)\mapsto\psi'(\frac{s}{t})$ with $\varphi,\psi\in\mathcal{C}^{\infty}_c((0,1),\rr)$. If $\psi\equiv 0$, the state moves along $+e_3$. If $\ph\equiv-\frac{\alpha}{2}\psi$, the state moves along $-e_3$, as 
		$$x_3(t)=-\frac{\alpha^2-4}{4}t^3\int_0^1\psi^2.$$
		\end{ex}
		The purpose of this article is to determine assumptions on the evaluation at 0 of the Lie brackets of $f_0$, $f_1$ and $f_2$ under which such a system is, in a sense, embedded within the system \eqref{affine-syst}. This analysis can be generalized to an affine system with integrators, as in the following example.
		\begin{ex}\label{exintegrateurintro}
			Let $k,k'\in\nn^*$ and $\alpha\in\rr$ be such that $|\alpha|\leq 2$. As in the previous case, the following system is not $L^{\infty}$-STLC
			$$\left\lbrace\begin{array}{rcl}
				x_1'&=&u\\x_2'&=&x_1\\&\vdots&\\x_k'&=&x_{k-1}\\	y_1'&=&v\\y_2'&=&y_1\\&\vdots&\\y_{k'}'&=&y_{k'-1}\\z_1'&=&x_k^2+y_{k'}^2+\alpha x_ky_{k'}\end{array}\right..$$
		\end{ex}
		
		In the first two examples, there was an explicitly signed direction due to the presence of a positive definite quadratic form. Here, we present a more general affine system that includes terms which are harmless with respect to STLC. This system captures the essence of the phenomenon we are about to study.
		\begin{ex}\label{excomplexe} We consider the following system
			\begin{equation}\label{art2jouetex}\left\lbrace\begin{array}{rcl}x_1'&=&u\\x_2'&=&x_1\\x_3'&=&v\\x_4'&=&\left(x_1^2+2x_3^2+\frac{1}{2}x_1x_3\right)-1028x_2^2-643vx_1^2-2vx_3\end{array}\right..\end{equation}
			First of all, for all $x,y\in\rr$, $x^2+2y^2+\frac 12 xy\geq \frac 34 \left(x^2+y^2\right)$. Let $T>0$. If $u\in L^1(0,T)$, we note $u_1:t\in[0,T]\mapsto\int_0^tu(s)\ds$. For all $u,v\in L^1(0,T)$, we then obtain
				\begin{equation}\label{art2jouetex1} 
				\int_0^T\left(x_1^2+2x_3^2+\frac{1}{2}x_1x_3\right)(t)\dt\geq \frac{3}{4}\int_0^T\left(x_1^2+x_3^2\right)(t)\dt=\frac{3}{4}\left\|(u_1,v_1)\right\|_{L^2}^2.
				\end{equation}
			Consequently, as in Examples \ref{firstexample} and \ref{exintegrateurintro}, the first three terms of the last line define a positive definite quadratic form. Let us show that the last three terms do not prevent us from drawing the conclusion. We have
			\begin{equation}\label{art2jouetex2}\left|\int_0^Tx_2(t)^2\mathrm{d}t\right|=\left|\int_0^T\left(\int_0^tu_1(s)\ds\right)^2\dt\right|\leq \int_0^T\left\|u_1\right\|_{L^2(0,t)}^2t\mathrm{d}t \leq \frac{T^2}{2}\left\|(u_1,v_1)\right\|_{L^2}^2,\end{equation}
			\begin{equation}\label{art2jouetex3} \left|\int_0^Tv(t)x_1(t)^2\mathrm{d}t\right|=\left|\int_0^Tv(t)u_1(t)^2\dt\right|\leq \left\|v\right\|_{L^{\infty}}\left\|u_1\right\|_{L^2}^2\leq \left\|(u,v)\right\|_{L^{\infty}}\left\|(u_1,v_1)\right\|_{L^2}^2.\end{equation}
			Then, the fourth and fifth terms are negligible compared to $\left\|(u_1,v_1)\right\|_{L^2}^2$ as $\left(T,\left\|(u,v)\right\|_{L^{\infty}}\right)\to 0$. Using \eqref{art2jouetex1}, \eqref{art2jouetex2} and \eqref{art2jouetex3}, an explicit integration of the solution from $0$ leads to
			\begin{align*}x_4(T)&\geq \left(\frac{3}{4}-514T^2-643\left\|(u,v)\right\|_{L^{\infty}}\right)\left\|(u_1,v_1)\right\|_{L^2}^2-x_3(T)^2\\&\geq C\left\|(u_1,v_1)\right\|_{L^2}^2-x_3(T)^2,\end{align*}
		for every $C \in \left(0, \frac{3}{4}\right)$, for small times $T$ and small controls in $L^{\infty}$.  As a result, any target in the set $\{x \in \mathbb{R}^4 \; ; \; x_4 + x_3^2 < 0\}$ is not reachable, so the system \eqref{art2jouetex} is not $L^{\infty}$-STLC.
		\end{ex}
		\subsection{Drift for proving obstructions}
	Our strategy to deny $W^{m,p}\times W^{m',p'}$-STLC consists in proving that system \eqref{affine-syst} has a drift. 
	\begin{defi}[Drift]\label{driftdefi}
		Let $e\in\rr^d$, $N\subset \rr^d$ be a vector subspace, $m,m'\in\nn$, $p,p'\in [1,+\infty]$, $\alpha\in\rr$ and $\Delta:L^1_{\mathrm{loc}}(\rr^+,\rr)^2\to\rr^+$. We say that system \eqref{affine-syst} has a drift along $e$ parallel to $N$ with strength $\Delta$ as $\left(t,t^{\alpha}\left\|(u,v)\right\|_{W^{m,p}\times W^{m',p'}}\right)\to 0$ when there exist $C>0$, $\beta>1$ and $\rho>0$ such that, for every $t\in(0,\rho)$ and $(u,v)\in W^{m,p}((0,t),\rr)\times W^{m',p'}((0,t),\rr)$ with $t^{\alpha}\left\|(u,v)\right\|_{W^{m,p}\times W^{m',p'}}\leq\rho$, 
		\begin{equation}\label{drift}\mathbb{P}x(t;(u,v))\geq C\Delta(u,v)-C\left|x(t;(u,v))\right|^{\beta},\end{equation}
		where $\mathbb{P}$ is a linear form, satisfying $\mathbb{P}(e)>0$ and $\restriction{\mathbb{P}}{N}\equiv 0$.
	\end{defi}
	\begin{lm}\label{lemmedrift}
		With the notations of Definition \ref{driftdefi}, if the system \eqref{affine-syst} has a drift along $e$ parallel to $N$ with strength $\Delta$ as $\left(t,t^{\alpha}\left\|(u,v)\right\|_{W^{m,p}\times W^{m',p'}}\right)\to 0$,
		then, the system \eqref{affine-syst} is not $W^{m,p}\times W^{m',p'}$-STLC.
	\end{lm}
	\begin{proof}The solution $x(t;(u,v))$ cannot reach targets of the form $x^*=-ae$ with $a>0$ small because this would entail
		$$-a\mathbb{P}(e)\geq C\Delta(u,v) -C\left|-ae\right|^{\beta} \geq -C\left|e\right|^{\beta}a^{\beta},$$ which is impossible  for $a$ small since $\beta > 1$ and $\mathbb{P}(e)>0$.
	\end{proof}
	\subsection{Algebraic background}
		We use the definitions and notations of Beauchard and Marbach in \cite{beauchard2024unified}.
	Let $X:= \{X_0,X_1, X_2\}$ be a set of 3 non-commutative indeterminates.
	\begin{defi}[Free algebra]
		\label{def:free-algebra}
		We note $\mathcal{A}(X)$ the free algebra generated by $X$ over the field $\rr$, 
		 \textit{i.e.}\ the unital associative algebra of polynomials of the indeterminates $X_0$, $X_1$ and $X_2$. 
	\end{defi}
	\begin{defi}[Free Lie algebra]	\label{def:free-lie-algebra}
		For $a,b\in\mathcal{A}(X)$, we define the Lie bracket of $a$ and $b$ as $[a,b]:= ab - ba$, also called $\text{ad}_a(b)$. We define by induction on $n\in\nn$, $\text{ad}_a^{n+1}(b)=[a,\text{ad}_a^n(b)]$.
		This operation is anti-symmetric and satisfies the Jacobi's identity: for all $a,b,c\in\mathcal{A}(X)$,
		\begin{equation}\label{jacobibase}[a,[b,c]]+[c,[a,b]]+[b,[c,a]]=0.\end{equation}
		Let $\mathcal{L}(X)$ be the free Lie algebra generated by $X$ over the field $\rr$, 
		 \textit{i.e.}\ the smallest linear subspace of $\mathcal{A}(X)$ containing $X$ and stable by the Lie bracket $[\cdot,\cdot]$.
	\end{defi}
	\begin{ex}
		$X_0^6X_1X_2^2X_0-X_1X_2^3\in\mathcal{A}(X)$ and $2[X_1,X_2]-7[X_0,[X_1,X_0]]\in\mathcal{L}(X)$.
	\end{ex}
	In order to unambiguously define the number of occurrences of an indeterminate within a (monomial) bracket, we introduce the following set of formal brackets.
	\begin{defi}[Iterated brackets]
		Let $\Br(X)$ be the free magma over $X$,  \textit{i.e.}\ the set of iterated brackets of elements of $X$, 
		defined by induction as: $X_0, X_1, X_2 \in \Br(X)$ and if $a, b \in \Br(X)$, then the ordered pair $(a,b)$ belongs to $\Br(X)$. 
		
		There is a natural evaluation mapping $\ev$ from $\Br(X)$ to $\mathcal{L}(X)$ defined by induction by 
		$\ev(X_i):= X_i$ for $i=0,1,2$ and $\ev((a, b)):= [\ev(a), \ev(b)]$. 
	\end{defi}
	\begin{defi}[Length and homogeneous layers within $\mathcal{L}(X)$] \label{coucheetoccuren} For $b\in \Br(X)$, $|b|$ denotes the length of $b$. For $i\in\llbracket 0,2\rrbracket$ and $b \in \Br(X)$, $n_i(b)$ denotes the number of occurrences of the indeterminate $X_i$ in $b$. We will use the notation: $n(b):=n_1(b)+n_2(b)=|b|-n_0(b)$. For $A_1,A_2\subset\nn$, $S_{A_1,A_2}(X)$ is the vector subspace of $\mathcal{L}(X)$ defined by
			\begin{equation}\label{art2sa1a2}
		S_{A_1,A_2}(X):= \spn\{\ev(b); \ b\in\Br(X), \ n_1(b)\in A_1, n_2(b)\in A_2\}.\end{equation}
		For $A\subset\nn$, $S_A(X)$ is defined by
		\begin{equation}\label{art2sa}S_A(X):= \spn\{\ev(b); \ b\in\Br(X), \ n(b)\in A\}.\end{equation}
		For $i\in\nn$, we write $S_i(X)$ instead of $S_{\{i\}}(X)$.
	\end{defi}
	\begin{exs}\begin{enumerate}\item[-]If $b:=(((X_1,(X_0,X_2)),X_2),(X_1,X_2))\in\Br(X)$, then $|b|=6$, $n_0(b)=1$, $n_1(b)=2$, $n_2(b)=3$ and $n(b)=5$.
			\item[-]We have $(X_1,X_1) \in\Br(X)$ and $\ev((X_1,X_1))=0$.
	\end{enumerate}
	\end{exs}
We have defined two notions of Lie brackets: one on vector fields -- see Definition \ref{def:lie-bracket-champ} -- and one on $\mathcal{A}(X)$ -- see Definition \ref{def:free-lie-algebra}. The aim of the following definition is to make them coexist.
	
	\begin{defi}[Evaluated Lie bracket]
		\label{Def:evaluated_Lie_bracket}
		Let $f_0, f_1,f_2$ be $\mathcal{C}^\infty(\Omega,\rr^d)$ vector fields on an open subset $\Omega$ of $\rr^d$ and  $f:=\{f_0,f_1,f_2\}$.
		For $b \in \mathcal{L}(X)$, we define $f_b:=\Lambda(b)$, where $\Lambda:\mathcal{L}(X) \to \mathcal{C}^\infty(\Omega,\rr^d)$ is the unique homomorphism of Lie algebras such that $\Lambda(X_i)=f_i$, for $i\in\llbracket0,2\rrbracket$.	When $b\in\Br(X)$, we will write $f_{b}$ instead of $f_{\ev(b)}$.
		For $\mathcal{N}\subset\mathcal{L}(X)$, we use the notation
		\begin{equation*}
			\mathcal{N}(f)(0):= \spn \{ f_{b}(0);\ b \in \mathcal{N} \} \subseteq \rr^d.
		\end{equation*}
	\end{defi}
		\begin{exs} \begin{enumerate}
			\item[-] If $b:=[[[X_1,X_2],X_0],[X_1,X_0]]$ one has $f_b=[[[f_1,f_2],f_0],[f_1,f_0]]$.
			\item[-] Let $\mathcal{N}:=\{X_1,X_2,[X_1,X_2]\}$. With the vector fields of system \eqref{jouet}, one has
				$$\mathcal{N}(f)(0)=\spn(f_{X_1}(0),f_{X_2}(0),f_{[X_1,X_2]}(0))=\spn(e_1,e_2),$$
				since $f_1$ and $f_2$ are constant vector fields and therefore all higher-degree Lie brackets vanish.
	\end{enumerate}\end{exs}

The vector space $\mathcal{A}(X)$ has a canonical basis (made of monomials $X_{i_1}^{j_1}\cdots X_{i_m}^{j_m}$ where $m\in\nn$, $i_1,\cdots,i_m\in\{0,1,2\}$ and $j_1,\cdots,j_m \in\nn^*$). This is not the case of $\mathcal{L}(X)$. In the free Lie algebra $\mathcal{L}(X)$, we have the concept of "Hall basis" see \cite[Theorem $1.2$]{Krob1987}. We propose in Proposition \ref{basedes2dansl} below an explicit algebraic basis of $S_2(X)$, which is a subset of a Hall basis of $\mathcal{L}(X)$ -- see Section \ref{sectionhallbasis} for details. For this statement, we need the following notation.
		\begin{nota}[Bracket integration $b0^\nu$] \label{def:0nu}
		For $b \in \mathcal{L}(X)$ and $\nu \in \nn$, we use the short-hand $b 0^\nu$ to denote the right-iterated bracket $[\dotsb[b, X_0], \dotsc, X_0]$, where $X_0$ appears $\nu$ times.
	\end{nota}
	\begin{ex} If $b:=[[X_1,X_2],[X_0,X_2]]$ then $b0^2=[[[[X_1,X_2],[X_0,X_2]],X_0],X_0]$.\end{ex}
	\begin{prop}\label{basedes2dansl} An algebraic basis of $S_2(X)$ is given by
	$$\left( W_{j,l}^i:=\left[X_i0^{j-1},X_i0^j\right]0^l\right)_{ j\in\nn^*,l\in\nn,i\in\{1,2\}}\cup\left( C_{j,l}:=(-1)^j\left[X_10^{\lfloor\frac{j+1}{2}\rfloor},X_20^{\lfloor\frac{j}{2}\rfloor}\right]0^l\right)_{ j,l\in\nn}.$$
	\end{prop}
	\begin{nota}
		When $l=0$, we will write $W_j^1,W_j^2,C_j$ instead of $W_{j,0}^1,W_{j,0}^2$ and $C_{j,0}$.
	\end{nota}
		\subsection{Main result}\label{sec-ex}
		\begin{defi}[Iterated primitives]\label{iteratedpri}
			For $j\in\nn$, $t>0$, we define by induction the iterated primitives of $u\in L^1((0,t),\rr)$, denoted $u_j:(0,t)\to\rr$ as $u_0:=u$ and $u_{j+1}(t):=\displaystyle\int_0^tu_j(s)\ds.$
		\end{defi}
	\begin{defi}[BC]\label{defbc}
		Let $e_1,e_2,e_3\in\rr^d$ be three vectors and $N\subset\rr^d$ a vector subspace. We say that $e_1,e_2,e_3,N$ satisfy (BC) if there exists a linear form $\mathbb{P}:\rr^d\to \rr$ such that $\restriction{\mathbb{P}}{N}\equiv0$ and $\mathbb{P}(e_3)^2<\mathbb{P}(e_1)\mathbb{P}(e_2).$
	\end{defi}
	\begin{rmq} The hypothesis $\mathbb{P}(e_3)^2<\mathbb{P}(e_1)\mathbb{P}(e_2)$ ensures that $\mathbb{P}(e_1)$ and $\mathbb{P}(e_2)$ have the same sign. Even if it means replacing $\mathbb{P}$ by $-\mathbb{P}$, we can assume that $\mathbb{P}(e_1)>0$, $\mathbb{P}(e_2)>0$, which will be assumed in the rest of the article. This implies in particular that $\mathbb{P}(e_1+e_2)>0$.
	\end{rmq}
	The purpose of this article is to prove the following theorem.
	\begin{theorem}\label{theoremdriftbis} \footnote{A generalization of this theorem to systems with $r\in\nn^*$ controls is presented in \cite[Theorem F.$0.1$]{gherdaouithese}.} Let $f_0,f_1,f_2$ be analytic vector fields over $\rr^d$ with $f_0(0)=0$.
		Let $k,m \in\nn^*$. We define the integer
		\begin{equation}\label{defpi}
			\pi(k,m):=1+\displaystyle\left\lceil\frac{2k}{m}\right\rceil,\end{equation}and the set \begin{equation}\label{ker}\begin{gathered}\mathcal{N}^m_k:={S}_{\llbracket 1,\pi(k,m)\rrbracket\setminus\left\lbrace 2\right\rbrace}(X)\\\cup\left\lbrace C_{j,l}; \ j\in\llbracket 0,2k-2\rrbracket,l\in\nn\right\rbrace\cup\left\lbrace W_{j,l}^1,W_{j,l}^2; \ j\in\llbracket 1,k-1\rrbracket,l\in\nn\right\rbrace,\end{gathered}\end{equation}
		where the last set in the right-hand side is empty if $k=1$. Assume that \begin{equation}\label{hypo}f_{W_k^1}(0), \quad f_{W_k^2}(0), \quad f_{C_{2k-1}}(0), \quad \mathcal{N}_k^m(f)(0) \quad \text{satisfy}\quad\text{(BC)}.\end{equation}Then, for all $p\in[1,+\infty]$, \eqref{affine-syst} has a drift along $f_{W_k^1}(0)+f_{W_k^2}(0)$ parallel to $\mathcal{N}_k^m(f)(0)$ with strength $\Delta :(u,v)\in L^1_{\mathrm{loc}}(\rr^+,\rr)^2\mapsto\displaystyle\int_0^t\left(u_k^2+v_k^2\right)\in\rr^+$ as $(t,t^{\alpha}\left\|(u,v)\right\|_{W^{m,p}})\to 0$ where $\alpha:=\frac{\pi(k,m)-2k}{\pi(k,m)-1}$. As a consequence, the system \eqref{affine-syst} is not $W^{m,p}$-STLC -- see Lemma \ref{lemmedrift}.
	\end{theorem}
		\begin{rmq}
		The assumption of Theorem \ref{theoremdriftbis} does not depend on $p$. We deny $W^{m,p}$-STLC for every $p\in[1,+\infty]$.
	\end{rmq}
		The parameter $k$ is associated with the drift order, $m,p$ with the regularity of the controls. The historical case $m=0$ can be studied using this theorem, for systems with an integrator, \textit{i.e.} satisfying $x_1'=u$. In general, the restriction to the case $m\geq 1$ is technical and could be improved in future work.		
		The brackets $W_k^1$ and $W_k^2$ are identified as bad brackets; they can create obstructions to STLC. More precisely, Theorem \ref{theoremdriftbis} highlights that, in $W^{m,p}$-STLC, the only brackets likely to compensate them are those in $\mathcal{N}_k^m$, or the bracket 
		$C_{2k-1}$ if its amplitude is sufficiently large.
		\begin{rmq}
			For example, if $k\in\nn^*$ and $m=1$, $\alpha=\frac{1}{2k}>0$. Thus, the controls $u,v$ need only to be bounded in $W^{1,p}$.
		\end{rmq}
		\begin{rmq}Note that, it may be possible to deal with techniques used in \cite[Section $10$]{beauchard2024unified} to remove  the assumption of analyticity of the vector fields $f_0$, $f_1$ and $f_2$. This will not be explored in the article.
		\end{rmq}
		\begin{ex}
			Let us go back to Example \ref{firstexample}. Let $m\geq 1$. One has
		$$f_{W_1^1}(0)=f_{W_1^2}(0)=2e_3, \quad f_{C_1}(0)=\alpha e_3, \quad \mathcal{N}^m_1(f)(0)=\spn\left(e_1,e_2\right).$$
			 Then, we consider the coordinate form $\mathbb{P}:(x_i)_{1\leq i \leq 3}\in\rr^3\mapsto x_3\in\rr$. Thus,
			$$\mathbb{P}\left(f_{W_1^1}(0)\right)=\mathbb{P}\left(f_{W_1^2}(0)\right)=2, \quad \mathbb{P}\left(f_{C_1}(0)\right)=\alpha, \quad \restriction{\mathbb{P}}{\mathcal{N}^m_1(f)(0)}\equiv 0.$$
			The condition $\mathbb{P}(e_3)^2<\mathbb{P}(e_1)\mathbb{P}(e_2)$ is verified if and only if $|\alpha|<2$. Under this hypothesis, one can apply Theorem \ref{theoremdriftbis} and \eqref{jouet} is not $W^{m,p}$-STLC, for every $m\geq 1$, $p\in[1,+\infty]$.
		\end{ex}
		\begin{ex}Let us go back to Example \ref{excomplexe}. Let $m\geq 1$. Since $f_{C_0}(0)=0$ and $f_0(0)=0$, all the iterated brackets vanish at $0$, \textit{i.e.}\ $f_{C_{0,l}}(0)=0$ for $l\in\nn$. Consequently,
			$$f_{W_1^1}(0)=2e_4, \quad f_{W_1^2}(0)=4e_4, \quad f_{C_1}(0)=\frac 12 e_4,\quad \mathcal{N}^m_1(f)(0)=\spn\left(e_1,e_2,e_3\right).$$
					Once again, the coordinate form yields the result.
		\end{ex}
		A major drawback of Theorem \ref{theoremdriftbis} is that its assumption is formulated as the existence of a linear form, which can make its practical application to a given system difficult. To address this issue, we have characterized the (non-)existence of such a linear form -- see Proposition \ref{bclink}. This leads to the following effective reformulation. One recalls that the set $\mathcal{N}_k^m$ is defined in \eqref{ker}.
	\begin{theorem}\label{theoremdrift} Let $f_0,f_1,f_2$ be analytic vector fields over $\rr^d$ with $f_0(0)=0$. Let $k,m \in\nn^*$.
		Let $\sigma:\rr^d\to\rr^d/\mathcal{N}_k^m(f)(0)$ be the canonical surjection, $\tilde{e}_1:=\sigma\left(f_{W_k^1}(0)\right)$, $\tilde{e}_2:=\sigma\left(f_{W_k^2}(0)\right)$ and $\tilde{e}_3:=\sigma\left(f_{C_{2k-1}}(0)\right)$.
		If the system \eqref{affine-syst} is $W^{m,p}$-STLC for a given $p\in[1,+\infty]$, one of the following conditions is satisfied
		\begin{enumerate}
			\item[\textbullet] $\tilde{e}_1=0$ or $\tilde{e}_2=0$,
			\item[\textbullet] $(\tilde{e}_1,\tilde{e}_2)$ is a linearly independent family and $\tilde{e}_3=a\tilde{e}_1+b\tilde{e}_2$ with $ab\geq \frac{1}{4}$,
			\item[\textbullet] $\tilde{e}_2=\beta\tilde{e}_1$ with $\beta<0$,
			\item[\textbullet] $\tilde{e}_2=\beta \tilde{e}_1$, $\tilde{e}_3=\gamma \tilde{e}_1$ with $\beta\leq\gamma^2$ and $\beta\neq 0$.
		\end{enumerate}
	\end{theorem}
	\begin{rmq}
		The canonical surjection $\sigma$ can be interpreted as a linear projection parallel to $\mathcal{N}_k^m(f)(0)$.
	\end{rmq}
	\begin{ex}  Let us return to Example \ref{firstexample}. One has
		$$f_{W_1^1}(0)=f_{W_1^2}(0)=2e_3, \quad f_{C_1}(0)=\alpha e_3,\quad \mathcal{N}^m_1(f)(0)=\spn\left(e_1,e_2\right),$$
		for any integer $m\in\nn^*$. Let $\sigma:\rr^3\to\rr^3/\mathcal{N}_1^m(f)(0)$ be the canonical surjection, $\tilde{e}_1=\sigma\left(f_{W_1^1}(0)\right)$, $\tilde{e}_2=\sigma\left(f_{W_1^2}(0)\right)$ and $\tilde{e}_3=\sigma\left(f_{C_1}(0)\right)$. Thus, $\tilde{e}_1=\tilde{e}_2=2$ and $\tilde{e}_3=\alpha$. The first three points of Theorem \ref{theoremdrift} are not satisfied. The last one is not verified if and only if $|\alpha|<2$ and then, by contraposition, Theorem \ref{theoremdrift} proves that, for every $m\in\nn^*$ and $p\in [1,+\infty]$, the system \eqref{jouet} is not $W^{m,p}$-STLC.
	\end{ex}
	\begin{ex} Let us focus on the following system
		\begin{equation}\label{exf1f2}\left\lbrace\begin{array}{rcl}x_1'&=&u\\x_2'&=&v\\x_3'&=&x_1v\\x_4'&=&x_3\\x_5'&=&\frac 12x_1^2+\frac 12x_2^2+x_4\end{array}\right..\end{equation}
		Let $\ep>0$ and $z\in\rr$. We define the controls
		$$u_z,v_z:t\in[0,4\ep]\mapsto\sqrt{|z|}\left(\mathbb{1}_{(0,\ep)}-\mathbb{1}_{(2\ep,3\ep)}\right)(t),\text{sgn}(z)\sqrt{|z|}\left(\mathbb{1}_{(\ep,2\ep)}-\mathbb{1}_{(3\ep,4\ep)}\right)(t),$$
	where $\mathbb{1}$ denotes the indicator function. One has $\left\|(u_z,v_z)\right\|_{L^{\infty}}\leq |z|^{\frac 12}$ and 
		$$x_1(4\ep)=x_2(4\ep)=0, \qquad x_3(4\ep)=z\ep^2,\qquad x_4(4\ep),x_5(4\ep)=\mathcal{O}\left(z\ep^3\right).$$
		Thus, $e_3=[f_1,f_2](0)$ is a second-order tangent-vector, in the sense of \cite{kawski2}. Noticing that $e_4=[[f_1,f_2],f_0](0)$, $e_5=[[[f_1,f_2],f_0],f_0](0)$, we deduce from \cite[Theorem $6$]{kawski3} that $e_4$ and $e_5$ are also tangent-vectors. Using \cite[Corollary $2.5$]{kawski2}, we obtain the $L^{\infty}$-STLC\ \footnote{In fact, \cite[Theorem $2.15$]{gherdaoui} proves that this system is $W^{m,\infty}_0$-STLC for every $m\in\nn$.} of system \eqref{exf1f2}.
		This example emphasizes the fact that we need to put not only the bracket $[X_1,X_2]$ but also $\{[X_1,X_2]0^{\nu}; \ \nu\in\nn\}$, in the set $\mathcal{N}_k^m$.
	\end{ex}
	\begin{ex} The necessary condition for STLC given by Theorem \ref{theoremdrift} is not sufficient,  \textit{i.e.}\ there are systems that are not $W^{m,p}$-STLC, but that verify at least one of the four points, for a fixed $m\in\nn^*$ and $p\in[1,+\infty]$. For instance, consider the following system
		\begin{equation}\label{ex-quartic}\left\lbrace\begin{array}{rcl} x_1'&=&u\\x_2'&=&v\\x_3'&=&x_1^4+x_2^4\end{array}\right..\end{equation}
		For every $k\in\nn^*$, $f_{W_k^1}(0)=0$. Thus, for all $k,m\in\nn^*$, this system satisfies the first point of Theorem \ref{theoremdrift}. However, for every $m\in\nn^*$ and $p\in[1,+\infty]$, this system is not $W^{m,p}$-STLC because $x_3'\geq 0$. In this situation, quadratic terms do not prevent the system from being controllable. Theorem \ref{theoremdrift} does not allow us to conclude for this example, since the drift arises from quartic terms. An opening to the case of quartic drifts is given in Section \ref{quartic}.
	\end{ex}
	\begin{ex}\label{exassymgeneral} Another counter-example is given by the following system
		\begin{equation}\label{exassym}\left\lbrace\begin{array}{rcl}x_1'&=&u\\x_2'&=&x_1\\x_3'&=&v\\x_4'&=&x_2^2+x_3^2\end{array}\right..\end{equation}Then,
		$f_{W_1^1}(0)=0$ and for all $ k\in\nn_{\geq 2}$, $f_{W_k^2}(0)=0.$
		Consequently, for all $k,m\in\nn^*$, the first point of Theorem \ref{theoremdrift} is satisfied by this system. However, for every $m\in\nn^*$ and $p\in[1,+\infty]$, the system is not $W^{m,p}$-STLC because $x_4'\geq 0$. Here, the obstacle to controllability is created by the bracket $W_2^1+W_1^2$. Indeed, Theorem \ref{theoremdrift} is designed to study competitions with quadratic brackets associated with controls of the same homogeneity in time. For this reason, we prove a generalization that allows us to deal with asymmetrical cases. The latter is addressed in Section \ref{sectionassym}.
	\end{ex}
	\subsection{The strategy to prove drifts}\label{heuristicstrat}
	\begin{defi}[Monomial basis]
		Let $\mathcal{B}$ be a basis of $\mathcal{L}(X)$. We say that $\mathcal{B}$ is monomial  if $\mathcal{B}\subset\ev(\Br(X))$. For such
		bases, if $b\in \mathcal{B}$, one can define $|b|$, $n_i(b)$ and $n(b)$ as in Definition \ref{coucheetoccuren} by importing these notions from $\Br(X)$. For $A\subset\nn$, we define $\mathcal{B}_A:=\{b\in\mathcal{B}; \ n(b)\in A\}$.
	\end{defi}
In this heuristic, we work with $\B$, an abstract monomial basis of $\mathcal{L}(X)$. Section \ref{sectionhallbasis} is dedicated to the construction of an appropriate one. We refer the interested reader to that section for further details.
	
	The proof of Theorem \ref{theoremdriftbis} is based on the Magnus-type representation formula and a strategy developed by Beauchard and Marbach in \cite{beauchard2024unified} to prove quadratic obstructions to STLC for single-input systems. Here are the main points to adapt their strategy to the case of multi-input systems. We fix $\mathcal{B}$, a monomial basis of $\mathcal{L}(X)$, $m\in\nn^*$ and $p\in[1,+\infty]$.
	The purpose is to create a drift -- see Definition \ref{driftdefi} -- as $\left(t,\left\|(u,v)\right\|_{W^{m,p}}\right)\to 0$. Let $M\in\nn^*$. The solution to \eqref{affine-syst} is given by the following formula -- see Proposition \ref{magnus} -- as $\left\|(u,v)\right\|_{L^1}\to0$,
	$$x(t;(u,v))=\mathcal{Z}_M(t;f,(u,v))(0)+\mathcal{O}\left(\left\|(u,v)\right\|_{L^1}^{M+1}+\left|x(t;(u,v))\right|^{1+\frac 1M}\right),$$
	where $\mathcal{Z}_M(t;f,(u,v))$ is an analytic vector field belonging to $S_{\llbracket 1,M\rrbracket}(f)$ and given by
	$$\mathcal{Z}_M(t;f,(u,v))=\sum_{b\in\mathcal{B}_{\llbracket 1,M\rrbracket}}\eta_b(t,(u,v))f_b,$$ where $\eta_b$ are functionals, called coordinates of the pseudo-first kind -- see \cite[Proposition $44$] {Beauchard_2023} for more details. These functionals are not easy to compute. However, the coordinates of the second kind $(\xi_b)_{b\in\mathcal{B}}$ -- see Definition \ref{coordinatesofthesecondkind} and \cite[Section $2.5.3.$]{Beauchard_2023} -- are straightforward to evaluate and there is a link between $\eta_b$ and $\xi_b$. Heuristically, we can think that $\eta_b\approx\xi_b$. Then, the Magnus formula becomes
	\begin{equation}\begin{gathered}\label{magnus-heuristique}x(t;(u,v))=\sum_{b\in\mathcal{B}_{\llbracket 1,M\rrbracket}}\xi_b(t,(u,v))f_b(0)+\sum_{b\in\mathcal{B}_{\llbracket 1,M\rrbracket}}\left(\eta_b-\xi_b\right)(t,(u,v))f_b(0)\\+\mathcal{O}\left(\left\|(u,v)\right\|_{L^1}^{M+1}+\left|x(t;(u,v))\right|^{1+\frac 1M}\right),\end{gathered}\end{equation} where the dominant part is the first sum. We 
	now consider $\mathfrak{b}_1,\mathfrak{b}_2,\mathfrak{b}_3\in\mathcal{B}_{\llbracket 1,M\rrbracket}$, brackets that will be used to create a drift and
	$\mathcal{N}\subset\mathcal{B}_{\llbracket 1,M\rrbracket}\setminus\{\mathfrak{b}_1,\mathfrak{b}_2,\mathfrak{b}_3\}$ the set defined as
	$$\forall b\in\mathcal{B}_{\llbracket 1,M\rrbracket}\setminus\{\mathfrak{b}_1,\mathfrak{b}_2,\mathfrak{b}_3\}, \quad \left(\forall i\in\llbracket 1,3\rrbracket, \  \xi_b\neq o(\xi_{\mathfrak{b_i}})\text{ as }\left(t,\left\|(u,v)\right\|_{W^{m,p}}\right)\to 0\right) \Rightarrow b\in\mathcal{N}.$$
	The elements of $\mathcal{N}$ are the brackets whose second-kind coordinates are not negligible compared to those of the brackets generating the drift.	As we are currently unable to handle these terms, we have chosen to include them in a vector space that will be removed by passing to a quotient. Then, we consider a linear form $\mathbb{P}$ on $\rr^d$, so that
	\begin{equation}\label{heuristique-p}\restriction{\mathbb{P}}{\mathcal{N}(f)(0)}\equiv 0\quad\text{and}\quad\displaystyle\Delta(u,v):=\sum_{i=1}^3\xi_{\mathfrak{b_i}}(t,(u,v))\mathbb{P}\left(f_{\mathfrak{b_i}}(0)\right)\geq 0.\end{equation}
	The existence of such a linear form $\mathbb{P}$ is not systematically guaranteed; this is the assumption (BC) of Theorem \ref{theoremdriftbis}.
	Finally, we fix $M$ (in terms of the parameters $m,p,\mathfrak{b}_1,\mathfrak{b}_2,\mathfrak{b}_3)$ and use interpolation inequalities to absorb the remainder term $\left\|(u,v)\right\|_{L^1}^{M+1}$ by the drift $ \Delta(u,v)$ and obtain
	\begin{equation}\label{heuristique-gn}\left\|(u,v)\right\|_{L^1}^{M+1}\lesssim \left\|(u,v)\right\|_{W^{m,p}}^{\alpha}\Delta(u,v),\end{equation}
	for some $\alpha>0$. Now, the formula \eqref{magnus-heuristique} becomes
	\begin{equation}\begin{gathered}\label{magnus-heuristique2}\mathbb{P}x(t;(u,v))=\Delta(u,v)+\sum_{\substack{b\in\mathcal{B}_{\llbracket 1,M\rrbracket}, \\b\notin\mathcal{N}\cup\{\mathfrak{b_1},\mathfrak{b}_2,\mathfrak{b}_3\}}}\xi_b(t,(u,v))\mathbb{P}\left(f_b(0)\right)\\+\sum_{\substack{b\in\mathcal{B}_{\llbracket 1,M\rrbracket},\\b\notin\mathcal{N}}}\left(\eta_b-\xi_b\right)(t,(u,v))\mathbb{P}\left(f_b(0)\right)+\mathcal{O}\left(\left\|(u,v)\right\|_{W^{m,p}}^{\alpha}\Delta(u,v)\right)+\mathcal{O}\left(\left|x(t;(u,v))\right|^{1+\frac 1M}\right).\end{gathered}\end{equation}
	Intuitively, as $(t,\left\|(u,v)\right\|_{W^{m,p}})\to0$, in the right-hand side of \eqref{magnus-heuristique2},
	\begin{enumerate}
		\item the second term is bounded by 
		$\ep\Delta(u,v)$, thanks to the choice of $\mathcal{N}$,
		\item the third term is small, as $\eta_b\approx\xi_b$,
		\item the fourth term is bounded by 
		$\ep\Delta(u,v)$, thanks to the asymptotic $\left\|(u,v)\right\|_{W^{m,p}}\to 0$,
		\item finally, the last term is part of the definition of a drift -- see Definition \ref{driftdefi}.
	\end{enumerate}
	Here, we pretend that a series of "small" terms keeps the same asymptotic, this fact is true in a precise framework defined in Section \ref{sectionbb}. The brackets in competition in our work are $\mathfrak{b}_1:=W_k^1$, $\mathfrak{b}_2:=W_k^2$ and $\mathfrak{b}_3:=C_{2k-1}$, for $k\in\nn^*$. 
	
		In the asymmetrical case, we use the same strategy as the one described above with a different truncation in the Magnus-type representation formula of the state -- see Proposition \ref{magnus-assym}.
	\subsection{State of the art}\hspace{0.5 cm}\par 
	The first known statement linking Lie brackets and small-time local controllability is proved by Hermann \cite{HERMANN1963325} and Nagano \cite{nagano}. These articles assert that the Lie Algebra Rank Condition is a necessary condition,  \textit{i.e.}\ if \eqref{affine-syst} is $L^{\infty}$-STLC, then, $\text{Lie}(f_0,f_1,f_2)(0):=\spn\{f_b(0); \ b\in\Br(X)\}=\rr^d$. This is the case for a control-affine system with an arbitrary number of controls. This condition is sufficient for systems satisfying $f_0\equiv 0$. This is proved by Chow in \cite{Chow} and Rashevski in \cite{rashevski} in $1938$-$39$. However, this condition is not sufficient in general when $f_0\not\equiv 0$. 
	
	Hermes and Sussmann proved in \cite{HERMES1982166,sussmann} in 1983 the following theorem. We recall that the set $S_A(X)$ for $A\subset \nn$ is defined in \eqref{art2sa}.
	\begin{theorem}
		Let $f_0,f_1,f_2$ be analytic vector fields over $\rr^d$ with $f_0(0) = 0$. Assume that the LARC is verified,  \textit{i.e.}\ $\text{Lie}(f_0,f_1,f_2)(0)=\rr^d$ and that,
		\begin{equation}\label{hermes} \forall k\in\nn^*, \quad {S}_{\llbracket 1,2k\rrbracket}(f)(0)\subseteq{S}_{\llbracket1,2k-1\rrbracket}(f)(0).
		\end{equation}
		Then, \eqref{affine-syst} is $L^{\infty}$-STLC.
	\end{theorem}
Hermes thus identifies that obstructions to STLC can only arise from brackets 
$b\in\Br(X)$ such that $n(b)$ is even.
	\subsubsection{For single-input systems}
	Sussmann was interested in the reciprocal of condition \eqref{hermes} in the case of single-input systems: is it necessary? Let us focus first on the case $k=1$. Assume that  there exists a bracket $b\in \mathcal{B}$ such that $f_{b}(0)\in{S}_{\llbracket 1,2\rrbracket}(f)(0)\setminus{S}_1(f)(0)$. The easiest bracket possible in our basis -- see Proposition \ref{basedes2dansl} -- is $b=W_1^1$. Then, in \cite{sussmann}, Sussmann proves the following necessary condition.
	\begin{theorem}\label{sussmann}
		Let $f_0,f_1$ be analytic vector fields over $\rr^d$ with $f_0(0) = 0$. If
		$x'=f_0(x)+uf_1(x)$
		is $L^{\infty}$-STLC, then $f_{W_1^1}(0)=[f_1,[f_1,f_0]](0)\in{S}_1(f)(0)$.
	\end{theorem}
	If $f_{W_1^1}(0)\in{S}_1(f)(0)$, we can now ask what about the bracket $W_2^1=(X_10,X_10^2)$.  We consider the system
	\begin{equation}\label{exsussman}\left\lbrace\begin{array}{rcl}
			x_1'&=&u\\x_2'&=&x_1\\x_3'&=&x_1^3+x_2^2\end{array}\right..\end{equation}
	For this system, 
	$${S}_1(f)(0)=\spn(e_1,e_2), \qquad f_{W_1^1}(0)=0, \qquad f_{W_2^1}(0)=2e_3.$$
	Thus, $f_{W_1^1}(0)\in{S}_1(f)(0)$ and $f_{W_2^1}(0)\notin{S}_1(f)(0)$. However, Sussmann proved in \cite{sussmann} that this system is $L^{\infty}$-STLC. The study of all quadratic drifts is therefore not obvious. Nevertheless, in \cite{beauchard2024unified}, Beauchard and Marbach propose a general method for demonstrating obstructions to the controllability of affine systems. This method is based on an adaptation of the Magnus formula -- see \cite{Beauchard_2023} -- which gives the expression of the solution to a control-affine system in the form of a series of Lie brackets. They use this method to study all quadratic drifts, proving the following statement.
	\begin{theorem}\label{theodriftmono}
		Let $f_0,f_1$ be analytic vector fields over $\rr^d$ with $f_0(0) = 0$ Let $m\in\llbracket-1,+\infty\llbracket$. If system  $x'=f_0(x)+uf_1(x)$ is $W^{m,\infty}$-STLC, then,
		\begin{equation}\label{general}
			\forall k\in\nn^*, \quad  f_{W_k^1}(0)\in{S}_{\llbracket 1,\pi(k,m)\rrbracket\setminus\{2\}}(f)(0),
		\end{equation}
		where
		$\pi(k, m):= 1 +\left\lceil\displaystyle\frac{2k-2}{m+1}\right\rceil$ with the convention $\pi(k,-1)=+\infty$ and $\pi(1,-1)=1$.
	\end{theorem}This is an extension of Theorem \ref{sussmann} ($k=1$, $m=0$). In particular, system \eqref{exsussman} is not $W^{1,\infty}$-STLC (apply Theorem \ref{theodriftmono} with $k=2$, $m=1$).
	\par Another necessary condition for controllability is Stefani's, in \cite{stefani}, which is concerned with the drift of the $2k$\textsuperscript{th} order term. The statement is the following one.
	\begin{theorem}\label{propstef}
		Let $f_0,f_1$ be analytic vector fields over $\rr^d$ with $f_0(0) = 0$. If system $x'=f_0(x)+uf_1(x)$ is $L^{\infty}$-STLC, then,
		\begin{equation}\label{stefani}
			\forall k\in\nn^*, \quad  \ad^{2k}_{f_1}(f_0)(0)\in S_{\llbracket 1,2k-1\rrbracket}(f)(0).\end{equation}
	\end{theorem}
	\subsubsection{For multi-input systems: the necessary condition of  \cite{refId0}}\label{stateofartrefId0}
	\par  In \cite{refId0}, Giraldi, Lissy, Moreau and Pomet consider affine systems \eqref{affine-syst} with $f_0(0)=f_2(0)=0$ (this class of systems is different from those studied in this article, because our assumptions imply that $f_2(0) \neq 0$ -- see Lemma \ref{libre}). They prove a necessary condition for $L^{\infty}$-STLC formulated on the bracket $W_1^1$.
	Their statement can be formulated as follow.
	\begin{theorem}\label{cngiraldilissy} 	Let $f_0,f_1,f_2$ be analytic vector fields over $\rr^d$ with $f_0(0) ={f_2(0)}= 0$.
		If \eqref{affine-syst} is $L^{\infty}$-STLC, then $f_{W_1^1}(0)\in S_{1,\nn}(f)(0)$.
	\end{theorem}
	In this framework, $f_{W_1^2}(0)=0$ and the equivalent of set $\mathcal{N}_k^m$ is $S_{1,\nn}(X)$. They prove a drift thanks to the coercivity of $\int u_1^2$, whereas in our setting, it results from a combination of two  terms. As a consequence, there is no bound on $n_2(b)$ is $S_{1,\nn}(X)$, they put brackets with associated coordinates of the second kind of any order on $v$ in the set $S_{1,\nn}(X)$.
	For the proof, they use the Chen-Fliess formula and re-organize its terms to form Lie brackets (as Stefani did in \cite{stefani}).
	
	In the same paper \cite{refId0}, the authors prove another necessary condition for STLC of multi-input control-affine systems. In simpler special cases, the theorem can be stated as follows.
	\begin{theorem}\label{giraldi2} 	Let $f_0,f_1,f_2$ be analytic vector fields over $\rr^d$ with $f_0(0) ={f_2(0)}= 0$. Assume that $f_{W_1^1}(0),f_{(X_1,(X_2,X_1))}(0),f_{((X_2,X_1),(X_0,X_1))}(0)\in S_{1,\nn}(f)(0)$. Let 
		\begin{equation*}\begin{gathered}q:(a_1,a_2)\in\rr^2\mapsto -a_1^2f_{W_2^1}(0)-a_2^2f_{((X_2,X_1),(X_2,X_1)0)}(0)\\-a_1a_2\left(f_{((X_2,X_1),M_2^1)}+f_{(M_1^1,(X_2,X_1)0)}\right)(0)\in\rr^d.\end{gathered}\end{equation*}
		If there exists a linear form $\ph: \rr^d\to\rr$, whose restriction to $S_{1,\nn}(f)(0)$ is zero and such that the quadratic form $(a_1,a_2)\in\rr^2\mapsto\ph\left(q(a_1,a_2)\right)$ is positive definite, then system \eqref{affine-syst} is not $W^{1,\infty}\times L^{\infty}$-STLC.
	\end{theorem}
Corollary \ref{bcquadratic} ensures that the hypothesis (BC) can be exactly formulated as the assumption of Theorem \ref{giraldi2}. Consequently, we use the same types of assumptions in Theorems \ref{theoremdriftbis} and \ref{theoremdriftbis2}
	\subsubsection{For multi-input systems: link with Sussmann's sufficient $\mathcal{S}(\theta)$-condition}
	\par Now, let us compare our necessary condition with Sussmann's sufficient $\mathcal{S}(\theta)$-condition -- see \cite[Theorem 7.3]{doi:10.1137/0325011} -- recalled in Theorem \ref{Prop:Sussmann} below -- see \cite[Theorem 3.29]{coronbook}.
	\begin{defi}
		The map $\sigma: \Br(X) \mapsto \mathcal{L}(X)$ is defined by $\sigma(b)=\ev(b)+\pi(\ev(b))$, where
		$\pi:\mathcal{L}(X) \mapsto \mathcal{L}(X)$ is the unique morphism of Lie algebra such that $\pi(X_0)=X_0$, $\pi(X_1)=X_2$ and $\pi(X_2)=X_1$.
	\end{defi} 
	For instance, $\sigma(W_k^1)=W_k^1+W_k^2$.
	\begin{theorem}[Sussmann's $\mathcal{S}(\theta)$-condition] \label{Prop:Sussmann}
		Let $f_0, f_1, f_2$ be analytic vector fields over $\rr^d$ with $f_0(0)=0$
		that satisfy the LARC, Lie$(f_0,f_1,f_2)(0)=\rr^d$.
		If there exists $\theta\in[0,1]$ such that, for every $\mathfrak{b} \in \Br(X)$ with $n_0(\mathfrak{b})$ odd and both $n_1(\mathfrak{b})$ and $n_2(\mathfrak{b})$ even,
		we have
		\begin{equation} \label{Hyp_Sussm}
			f_{\sigma(\mathfrak{b})}(0) \in \spn\{ f_b(0); \ b \in \Br(X), \ n(b)+\theta n_0(b)<n(\mathfrak{b}) +\theta n_0(\mathfrak{b})\}
		\end{equation}
		then \eqref{affine-syst} is $L^\infty$-STLC.
	\end{theorem}
	Sussmann’s $\mathcal{S}(\theta)$-condition is a very popular sufficient condition for controllability of affine systems because it is quite
	simple to apply. Nevertheless, more powerful conditions are known, for example Agrachev
	and Gamkrelidze’s -- see \cite[Theorem $4$]{agrachev} -- or Krastanov's -- see \cite[Theorem $2.7$]{Krastanov}.

	Our necessary condition also involves the Lie brackets $W_k^1$ and $W_k^2$. If $m=1$ and the hypothesis \eqref{hypo} of Theorem \ref{theoremdriftbis} is verified, then, for every $\theta\in[0,1]$, \eqref{Hyp_Sussm} does not hold for $\mathfrak{b}=W_k^1$ -- see Appendix \ref{append-linkstheta} for a proof.\par
	Moreover, for the system \eqref{jouet}, we have $f_{\sigma(W_1^1)}(0)=4e_3$ and $\mathcal{N}_1^1(f)(0)=\spn(e_1,e_2)$. Furthermore, for every $\theta\in[0,1]$, 
	$$\spn(f_b(0), \ b\in\Br(X), \ n(b)+\theta n_0(b)<2+\theta)\subseteq\mathcal{N}_1^1(f)(0).$$
	Consequently, $W_1^1$ does not satisfy the condition \eqref{Hyp_Sussm}. However, as already explained in Section \ref{sec-ex}, the $W^{1,p}$-STLC of \eqref{jouet} is function of the value of $\alpha$. The Sussmann's $\mathcal{S}(\theta)$-condition is sensitive to the direction of the Lie brackets, but not to their amplitude.
	\subsubsection{For multi-input systems:  the necessary condition of  \cite{lewis}}\label{stateofartlewis}
	In \cite{lewis}, Lewis and Hirschorn use the Chen–Fliess series to prove the following result, which we restate using
	our slightly different context and our own notations.
	\begin{theorem}\label{lewis} Let $f_0,f_1,f_2$ be analytic vector fields with $f_0(0)=0$ and $(f_1(0),f_2(0))$ linearly independent.
		Assume that $0$ is a regular point for $\text{Lie}(f_1,f_2)(0)$\footnote{\textit{i.e.}\ $\dim\{g(x); g \in\text{Lie}(f_1,f_2)\}$ does not depend on $x$ on a neighborhood of $0$}. Let $N:=S_1(f)(0)+ \spn\{\ad^{\mu}_{f_0}([f_1,f_2])(0); \ \mu\in\nn\}+ \text{Lie}(f_1,f_2)(0)$. Assume that the system \eqref{affine-syst} is $L^{\infty}$-STLC. Let $\sigma: \rr^d\to\rr^d/N$ be the canonical surjection. Let us define the vector-valued quadratic form
		$$B_N : (a_1,a_2)\in\rr^2\mapsto a_1^2\sigma(f_{W_1^1}(0)) + a_2^2\sigma(f_{W_1^2}(0)) + 2a_1a_2\sigma(f_{C_1}(0))\in\rr^d/N.$$
		Then, there does not exist a linear form $\mathbb{P}:\rr^d/N\to \rr$ such that the quadratic form $(a_1,a_2)\in\rr^2\mapsto\mathbb{P}\left(B_N(a_1,a_2)\right)\in\rr$ is a positive definite quadratic form.
	\end{theorem}
	In \cite{lewis}, the authors focus on the case where $k=1$, whereas we deal with the case of $k\in\nn^*$. They prove a necessary condition for $L^{\infty}$-STLC of control-affine systems without any assumption of smallness on the controls. In this article, we work in the functional framework $W^{m,p}$-STLC, with $m\in\nn^*$ and $p\in[1,+\infty]$. They require the compensation of all the elements of $\text{Lie}(f_1,f_2)(0)$, while we can use the smallness of the controls. However in our low regularity regime $k=m=1$, we require compensation for all brackets of $S_3$, which is not the case in Theorem \ref{lewis}. Enlightened by \cite{lewis}, it might be possible te refine our result.
	
	Other authors have studied obstructions to controllability linked with quadratic phenomenons -- see \textit{e.g.}, \cite{agrachev2,RogerBrockett2013}.
	\subsection{Structure of the article}
	The paper is organized as follow: in Section $2$, we present some tools and properties that will be used in the next sections.
	In Section $3$, we give the proof of Theorem \ref{theoremdriftbis}. Finally, in Section $4$, we prove a generalization of this result in the asymmetrical case -- see Theorem \ref{theoremdriftbis2}. Some elements of proof are developed in the appendix.
		\section{Requirements for the proof}
	\subsection{Hall sets and bases of $\mathcal{L}(X)$}\label{sectionhallbasis}
	The purpose of this section is to introduce tools for constructing a convenient basis of the free Lie algebra $\mathcal{L}(X)$.
	\begin{defi}[Left and right factors] For $b\in\Br(X)$ with 
		$|b|> 1$, $b$ can be written in a unique way as $b = (b_1, b_2)$, with $b_1, b_2\in\Br(X)$. We use the notations
		$\lambda(b) = b_1$ and $\mu(b) = b_2$, which define maps $\lambda,\mu :\Br(X)\setminus X\to\Br(X)$.
	\end{defi}
	\begin{ex} If $b:=((X_1,X_2),((X_2,X_0),X_0))$, we have $\lambda(b)=(X_1,X_2)$ and $\mu(b)=((X_2,X_0),X_0)$. \end{ex}
	\begin{defi}[Hall set]
		A \emph{Hall set} is a subset $\B$ of $\Br(X)$ endowed with a total order $<$ such that
		\begin{itemize}
			\item $X \subset \B$,
			\item for all $b_1, b_2 \in \Br(X)$, $(b_1, b_2) \in \B$ if and only if $b_1, b_2 \in \B$, $b_1 < b_2$ and either $b_2 \in X$ or $\lambda(b_2) \leq b_1$, 
			\item for all $b_1, b_2 \in \B$ such that $(b_1,b_2) \in \B$ then $b_1 < (b_1,b_2)$.
		\end{itemize}
	\end{defi}
This definition may appear mysterious at first glance, but its main usefulness lies in the following theorem, due to Viennot -- see \cite{Krob1987}.
	\begin{theorem}
		\label{thm:viennot}
		Let $\B \subset \Br(X)$ be a Hall set. 
		Then $\ev(\B)$ is a basis of $\mathcal{L}(X)$.
	\end{theorem}
	The goal is now to construct a Hall set in order to obtain a basis for the free Lie algebra $\mathcal{L}(X)$. To simplify the formulas that follow, we will use the following notations.
	\begin{notas} \begin{enumerate}
		\item[-]	For $\B$ a Hall set and $A \subset \mathbb{N}$, we denote by $\B_A$ the subset of $\B$ defined by
		$\B_A:=\{ b \in \B; \ n(b) \in A \}$. We write $\B_i$ instead of $\B_{\{i\}}$. 
		\item[-] 	As for $\mathcal{L}(X)$, we define $b 0^\nu$ for $b \in \Br(X)$ and $\nu \in \nn$ as $(\dotsb(b, X_0), \dotsc, X_0)$, where $X_0$ appears $\nu$ times.
		\end{enumerate}
	\end{notas}
	The definition of a Hall set is also an algorithm for its construction.
	Indeed, the subsets $\B_N$ of a Hall set $\B$ can be constructed by induction on $N$. 
	We may start, for example,  with $\B_0=\{X_0\}$ and 
	$\B_1=\{ X_1 0^{\nu_1}, X_2 0^{\nu_2}; \ \nu_1, \nu_2 \in \mathbb{N}\}$ with the following order
	$$\forall k\in \mathbb{N}, \qquad	X_10^k<X_20^k<X_10^{k+1}<X_20^{k+1}<\cdots<X_0.$$
	which is compatible with the 3 axioms above. 
	For $N \geq 2$, to find all Hall elements $b \in \B_N$ given $\B_{\llbracket 1 , N-1 \rrbracket}$,
	one adds first all $(a, b)$ with $a \in \B_{N-1}$, $b \in X$ and $a < b$. 
	Then for each bracket $b = (b_1, b_2) \in \B_{\llbracket 1 , N-1 \rrbracket}$,
	one adds all the $(a, b)$ with $a \in \B_{N-n(b)}$ and $b_1 \leq a < b$. 
	Finally, one inserts the newly generated elements of $\B_N$ into an ordering, maintaining the condition that $a < (a, b)$.
	\begin{nota} For $b=(b_1,b_2)\in\Br(X)$ and $j\in\nn$, we use the notation $(-1)^jb$ as $(-1)^jb=(b_2,b_1)$ if $j$ is odd and $(-1)^jb=b$ if $j$ is even. This notation is unconventional, it is used to condense the writing.
	\end{nota}
	With this construction, we obtain the following statement, already used in \cite[Proposition $2.13$]{gherdaoui}.
	\begin{prop}\label{Prop:Bs} There exists a Hall set $\mathcal{B}$ such that $X_0$ is maximal, 
		\begin{equation}\label{base-s1}\mathcal{B}_1=\left\lbrace M_j^1:=X_10^j, \quad M_j^2:=X_20^j;\quad j\in\nn\right\rbrace,\end{equation}
		and $\mathcal{B}_2=\mathcal{B}_{2,good}\cup\mathcal{B}_{2,bad}$ with 
		\begin{equation}\label{base-s2bad}\mathcal{B}_{2,bad}=\left\lbrace W_{j,l}^1:=(M_{j-1}^1,M_j^1)0^l,\quad W_{j,l}^2:=(M_{j-1}^2,M_j^2)0^l;\quad j\in\nn^*,l\in\nn\right\rbrace,\end{equation} and
		\begin{equation}\label{base-s2good}\mathcal{B}_{2,good}=\left\lbrace C_{j,l}:=(-1)^j\left(M_{\lfloor\frac{j+1}{2}\rfloor}^1,M_{\lfloor\frac{j}{2}\rfloor}^2\right)0^l;\quad j,l\in\nn\right\rbrace.\end{equation}
	Moreover, to avoid cluttering the formulas, all these symbols will indifferently denote either the elements of $\Br(X)$ themselves or their evaluation by $\ev$ in $\mathcal{L}(X)$, for example in Proposition \ref{basedes2dansl}.
	\end{prop}
	\begin{nota}
		When $l=0$, we will write $W_j^1,W_j^2,C_j$ instead of $W_{j,0}^1,W_{j,0}^2$ and $C_{j,0}$.
	\end{nota}
	\begin{rmq}
		This basis will play a central role in what follows. In particular, Proposition \ref{basedes2dansl} -- previously stated but whose proof was postponed -- is a direct consequence of this result.
	\end{rmq}
	\subsection{Expression of the coordinates of the second kind}
We recall that the iterated primitives $u_j$ of $u\in L^1((0,t),\rr)$ are specified in Definition \ref{iteratedpri}.
	\begin{defi}[Coordinates of the second kind]\label{coordinatesofthesecondkind}
		Let $\mathcal{B} \subset \Br(X)$ be a Hall set.
		The coordinates of the second kind associated with $\B$ is the unique family
		$(\xi_b)_{b\in\B}$ of functionals $\xi_b:\rr_+ \times L^1_{\mathrm{loc}}(\rr_+,\rr)^2 \to \rr$ defined by induction as: 
		for every $t>0$ and $u, v \in L^1((0,t),\rr)$,
		\begin{itemize}
			\item $\xi_{X_0}(t,(u,v)):= t,$ $\xi_{X_1}(t,(u,v)):=u_1(t)$ and $\xi_{X_2}(t,(u,v)):= v_1(t)$,
			\item for $b \in \mathcal{B} \setminus X$, there exists  a unique couple $(b_1,b_2)$ of elements of $\mathcal{B}$ such that $b_1<b_2$ and a unique maximal integer $m\geq 1$ with $b=\ad_{b_1}^m (b_2)$ and then
			\begin{equation}
				\xi_{b}(t,(u,v)):=\frac{1}{m!} \int_0^t  \xi_{b_1}^m(s,(u,v)) {\xi}'_{b_2}(s,(u,v)) \dd s.
			\end{equation}
		\end{itemize}
	\end{defi}

 The following proposition is taken from \cite[Proposition $2.18$]{gherdaoui}.
	\begin{prop}The following equalities hold.
		\begin{enumerate}
			\item 
			For $b\in\mathcal{B}$ and $\nu\in\nn$, 
			\begin{equation}\label{fonctcrochetgauche}\xi_{b0^{\nu}}(t,(u,v))=\int_0^t\frac{(t-s)^{\nu}}{\nu !}\xi_b'(s,(u,v))\ds.\end{equation}
			\item For every $j\in\nn$,
			\begin{equation}\label{egbase-s11}\xi_{M_j^1}(t,(u,v))=u_{j+1}(t),\end{equation}
			\begin{equation}\label{egbase-s12}\xi_{M_j^2}(t,(u,v))=v_{j+1}(t).\end{equation}
			\item For every $j\in\nn^*$, $l\in\nn$,
			\begin{equation}\label{egbase-s2bad1}\xi_{W_{j,l}^1}(t,(u,v))=\frac{1}{2}\int_0^t\frac{(t-s)^l}{l!}u_j^2(s)\ds,\end{equation}
			\begin{equation}\label{egbase-s2bad2}\xi_{W_{j,l}^2}(t,(u,v))=\frac{1}{2}\int_0^t\frac{(t-s)^l}{l!}v_j^2(s)\ds.\end{equation}
			\item For every $j,l\in\nn$,
			\begin{equation}\label{egbase-s2good}\xi_{C_{j,l}}(t,(u,v))=\int_0^t\frac{(t-s)^l}{l!}u_{\left\lfloor\frac{j}{2}\right\rfloor+1}(s)v_{\left\lfloor\frac{j+1}{2}\right\rfloor}(s)\ds.\end{equation}
		\end{enumerate}
	\end{prop}
	\begin{proof}
		The first three points are proved in \cite[Lemma $3.6$ and Proposition $3.7$]{beauchard2024unified}. Let us prove the last one: let $j,l\in\nn$. Using the first point,
		\begin{equation*}\label{fonct}\xi_{C_{j,l}}(t,(u,v))=\int_0^t\frac{(t-s)^l}{l!}\xi'_{C_j}(s,(u,v))\ds.\end{equation*}
		First, assume that $j=2j_0$ is even. As $C_j=\ad_{X_10^{j_0}}^1(X_20^{j_0})$, Definition \ref{coordinatesofthesecondkind},  \eqref{egbase-s11} and \eqref{egbase-s12} give
		\begin{equation*}\label{fonct2}\xi_{C_j}(t,(u,v))=\int_0^tu_{j_0+1}(s)v_{j_0}(s)\mathrm{d}s=\int_0^tu_{\lfloor\frac{j}{2}\rfloor+1}(s)v_{\lfloor\frac{j+1}{2}\rfloor}(s)\ds.\end{equation*}
An analogous argument yields the result when $j$ is odd.
	\end{proof}
Given a function $h$ that is $\mu$-times differentiable, we denote its $\mu$-th derivative by $h^{(\mu)}$.

	\begin{lm}\label{coordinates-ipp}
		Let $j,l,N\in\nn$ be such that $N\leq\lfloor\frac{j+1}{2}\rfloor-1$. Let $t>0$ and $u,v\in L^1((0,t),\rr)$. Then, 
		\begin{equation}\begin{gathered}\xi_{C_{j,l}}(t,(u,v))=\sum_{\mu=0}^N(-1)^{\mu}u_{\lfloor\frac j2\rfloor+\mu+2}(t)\left(v_{\lfloor\frac{j+1}{2}\rfloor}\frac{(t-\cdot)^l}{l!}\right)^{(\mu)}(t)+\\(-1)^{N+1}\int_0^tu_{\lfloor\frac j2\rfloor+N+2}(s)\frac{\mathrm{d}^{N+1}}{\mathrm{d}s^{N+1}}\left(v_{\lfloor\frac{j+1}{2}\rfloor}(s)\frac{(t-s)^l}{l!}\right)\ds.\end{gathered}\end{equation}
	\end{lm}
	\begin{proof}We prove the lemma by induction on $N$, using an integration by parts.\end{proof}
	\subsection{Estimates on the coordinates of the second kind}
	\begin{lm}\label{holder} Let $p\in[1,+\infty]$ and $j_0\in\nn^*$. For every $j\geq j_0$, $t>0$, $u\in L^1((0,t),\rr)$, 
		\begin{equation}\label{injec}\left\|u_j\right\|_{L^p}\leq \frac{t^{j-j_0}}{(j-j_0)!}\left\|u_{j_0}\right\|_{L^p}.	\end{equation}
	\end{lm}
	\begin{proof}
		This lemma is proved in \cite[Lemma A.$6$]{beauchard2024unified}.
	\end{proof}
	\begin{prop}\label{estimationcoordonnées} The following inequalities hold.
		\begin{enumerate}
			\item Let $p\in[1,+\infty]$ and $j_0\in\nn^*$. There exists $c>0$ such that, for every $j\geq j_0$, $t>0$, $u,v\in L^1((0,t),\rr)$,
			\begin{equation}\label{estbase-s11}
				\left|\xi_{M_j^1}(t,(u,v))\right|\leq \frac{(ct)^{|M_j^1|}}{|M_j^1|!}t^{-(j_0+1)}t^{1-\frac 1p}\left\|u_{j_0}\right\|_{L^p},
			\end{equation}
			\begin{equation}\label{estbase-s12}
				\left|\xi_{M_j^2}(t,(u,v))\right|\leq \frac{(ct)^{|M_j^2|}}{|M_j^2|!}t^{-(j_0+1)}t^{1-\frac 1p}\left\|v_{j_0}\right\|_{L^p}.
			\end{equation}
			\item Let $p\in[1,+\infty]$ and $j_0\in\nn^*$. There exists $c>0$ such that, for every $j\geq j_0$, $l\in\nn$, $t>0$, $u,v\in L^1((0,t),\rr)$,
			\begin{equation}\label{estbase-s2bad1}
				\left|\xi_{W_{j,l}^1}(t,(u,v))\right|\leq \frac{(ct)^{|W_{j,l}^1|}}{|W_{j,l}^1|!}t^{-(2j_0+1)}t^{1-\frac 1p}\left\|u_{j_0}\right\|_{L^{2p}}^2,
			\end{equation}
			\begin{equation}\label{estbase-s2bad2}
				\left|\xi_{W_{j,l}^2}(t,(u,v))\right|\leq \frac{(ct)^{|W_{j,l}^2|}}{|W_{j,l}^2|!}t^{-(2j_0+1)}t^{1-\frac 1p}\left\|v_{j_0}\right\|_{L^{2p}}^2.
			\end{equation}
			\item Let $p,q\in[1,+\infty]$ be such that $\frac 1p+\frac 1q\leq 1$ and $k,k'\in\nn^*$ with $k'\leq k$. There exists $c>0$ such that, for every $j\geq k+k'-1$, $l\in\nn$, $0<t<1$, $u,v\in L^1((0,t),\rr)$,
			\begin{equation}\label{estbase-s2good1}\begin{gathered}
					\left|\xi_{C_{j,l}}(t,(u,v))\right|\leq \frac{(ct)^{|C_{j,l}|}}{|C_{j,l}|!}t^{-(1+k+k')}t^{1-\left(\frac 1p+\frac 1q\right)}\left\|u_k\right\|_{L^p}\left\|v_{k'}\right\|_{L^q}\\+\mathbb{1}_{l\leq k-2-\lfloor\frac j2\rfloor}K\left(\sum_{\mu=1}^k|u_{\mu}(t)|^2+t\left\|v_{k'}\right\|_{L^2}^2\right),\end{gathered}
			\end{equation}
			where $K$ only depends on $k$ and $\mathbb{1}$ denotes the indicator function. In other words, if $l> k-2-\lfloor\frac j2\rfloor$, there is no boundary term.
		\end{enumerate}
	\end{prop}
	\begin{proof}
		The first two points are proved in \cite[Proposition $3.10$]{beauchard2024unified}. We prove the last one: let $j\geq k+k'-1$, $l\in\nn$, $0<t<1$, $u,v\in L^1((0,t),\rr)$. 
		
		\textit{First, assume that $\lfloor\frac{j}{2}\rfloor+1\geq k$ (this is always the case when $k=k'$)}.
		Using \eqref{egbase-s2good} and Hölder's inequality, we obtain
		$$\left|\xi_{C_{j,l}}(t,(u,v))\right|\leq\frac{t^l}{l!}t^{1-\left(\frac 1p+\frac 1q\right)}\left\|u_{\lfloor\frac j2\rfloor+1}\right\|_{L^p}\left\|v_{\lfloor\frac{j+1}{2}\rfloor}\right\|_{L^q}.$$
		Notice that $\lfloor\frac{j+1}{2}\rfloor\geq\lfloor\frac{k+k'}{2}\rfloor\geq k'$.
		Using two times Lemma \ref{holder} with $j_0=k$ and with $j_0=k'$, we get
		$$\left|\xi_{C_{j,l}}(t,(u,v))\right|\leq\frac{t^l}{l!}\frac{t^{\lfloor\frac{j}{2}\rfloor+1-k}}{(\lfloor\frac{j}{2}\rfloor+1-k)!}\frac{t^{\lfloor\frac{j+1}{2}\rfloor-k'}}{(\lfloor\frac{j+1}{2}\rfloor-k')!}t^{1-\left(\frac 1p+\frac 1q\right)}\left\|u_k\right\|_{L^p}\left\|v_{k'}\right\|_{L^q}.$$
		As $\lfloor\frac{j+1}{2}\rfloor+\lfloor\frac{j}{2}\rfloor=j$, we obtain the result because, for all $j\geq k+k'-1$, $l\in\nn$, 
		\begin{align*}\frac{1}{l!(\lfloor\frac{j}{2}\rfloor+1-k)!(\lfloor\frac{j+1}{2}\rfloor-k')!}&=\binom{l+j+2}{l}\binom{j+2}{\lfloor\frac{j}{2}\rfloor+1-k}\binom{\lfloor\frac{j+1}{2}\rfloor+k+1}{\lfloor\frac{j+1}{2}\rfloor-k'}\frac{(k+k'+1)!}{(l+j+2)!}\\&\leq \frac{2^{3|C_{j,l}|}(k+k'+1)!}{(l+j+2)!},\end{align*}
		where $|C_{j,l}|=j+l+2$.

		\textit{Now, we assume that $\lfloor\frac{j}{2}\rfloor\leq k-2$}. Let $N:=k-2-\lfloor\frac{j}{2}\rfloor\geq 0$. Note that $N\leq\lfloor\frac{j+1}{2}\rfloor-1$. We use Lemma \ref{coordinates-ipp} to obtain $\xi_{C_{j,l}}(t,(u,v)=A+B$ with $A$ the boundary terms and $B$ the integral part. Using Leibniz formula, we get
		$$\left(v_{\lfloor\frac{j+1}{2}\rfloor}\frac{(t-\cdot)^l}{l!}\right)^{(\mu)}(t)=\left\lbrace\begin{array}{cc}0&\text{if }\mu<l\\\displaystyle\binom{\mu}{l}(-1)^lv_{\lfloor\frac{j+1}{2}\rfloor+l-\mu}(t)&\text{otherwise}\end{array}\right..$$
		Consequently, the following inequality holds
		$$|A|\leq\mathbb{1}_{l\leq N}\sum_{\mu=l}^N\binom{\mu}{l}\left|u_{\lfloor\frac{j}{2}\rfloor+\mu+2}(t)v_{\lfloor\frac{j+1}{2}\rfloor+l-\mu}(t)\right|.$$ 
		Then, using Young's and Cauchy--Schwarz's inequality,
		$$|A|\leq\mathbb{1}_{l\leq N}2^{N-1}\left(\sum_{\mu=1}^k\left|u_{\mu}(t)\right|^2 +t\underset{\mu\in\llbracket l,N\rrbracket}\max\left\|v_{\lfloor\frac{j+1}{2}\rfloor+l-\mu-1}\right\|_{L^2}^2\right),$$
		as $1\leq \lfloor\frac{j}{2}\rfloor+\mu+2\leq k$, for $\mu\in\llbracket 1,N\rrbracket$. Note that $\lfloor\frac{j+1}{2}\rfloor+l-\mu-1\geq k'$. Finally, applying Lemma \ref{holder} with $p=2$ and $j_0=k'$ and using $t\in(0,1)$, we obtain the following estimate
		\begin{equation}\label{coordinates-ipp-a}|A|\leq\mathbb{1}_{l\leq N}2^{k-3}\left(\sum_{\mu=1}^k\left|u_{\mu}(t)\right|^2 +t\left\|v_{k'}\right\|_{L^2}^2\right).\end{equation}
		We finally estimate $B$. Using Leibniz formula and Hölder's inequality, we get
		$$|B|\leq \sum_{\mu=0}^{\min(N+1,l)}\binom{N+1}{\mu}\frac{t^{l-\mu}}{(l-\mu)!}t^{1-\left(\frac1p+\frac 1q\right)}\left\|u_k\right\|_{L^p}\left\|v_{j+\mu+1-k}\right\|_{L^q}.$$
		Using Lemma \ref{holder} with $p=q$ and $j_0=k'$, we have
		\begin{equation}\label{coordinates-ipp-b}
			|B|\leq \frac{(2^kt)^{|C_{j,l}|}}{l!}t^{-(1+k+k')}t^{1-\left(\frac 1p+\frac 1q\right)}\left\|u_k\right\|_{L^p}\left\|v_{k'}\right\|_{L^q}.
		\end{equation}
		Finally, for all $l\geq 0$, for all $j\in\llbracket 0,2k-3\rrbracket$, 
		\begin{equation}\label{coordinates-ipp-c}\frac{1}{l!}= \frac{(j+2)!}{(j+l+2)!}\binom{j+l+2}{l}\leq \frac{(2k-1)!}{(j+l+2)!}2^{j+l+2}\leq \frac{(2(2k-1)!)^{|C_{j,l}|}}{(j+l+2)!}.\end{equation}
		Thus,  equations \eqref{coordinates-ipp-a}, \eqref{coordinates-ipp-b} and \eqref{coordinates-ipp-c} lead to the desired inequality.
	\end{proof}
	\subsection{Analytic norms}
	The following paragraph is inspired by \cite[Section $4.1$]{beauchard2024unified}. We introduce some basic notions about analytic vector fields and norms of analytic vector fields. These will be useful for ensuring the convergence of the series that we will consider in the following sections.
	\begin{defi}[Length and factorial of a multi-index, partial derivative]Let $d\in\nn^*$ be a positive integer and $\alpha=(\alpha_1,\cdots,\alpha_d)\in\nn^d$ be a multi-index. We define
		\begin{enumerate}
			\item the length of $\alpha$ as: $|\alpha|:=\alpha_1+\cdots+\alpha_d$,
			\item the factorial of $\alpha$ as: $\alpha!:=\alpha_1!\times\cdots\times\alpha_d!$,
			\item  the partial derivative: $\partial_{\alpha}:= \partial_{x_1}^{\alpha_1}\cdots\partial_{x_d}^{\alpha_d}$.
		\end{enumerate} 
	\end{defi}
	\begin{defi}[Analytic vector fields, analytic norms] Let $\delta >0$ and $\overline{B}_{\delta}$ be the closed ball of radius $\delta$, centered at $0\in\rr^d$. For $r>0$, we define $C^{\omega,r}\left(\overline{B}_{\delta};\rr^d\right)$ as the subspace of analytic vector fields $f$ on an open neighborhood of $\overline{B}_{\delta}$, for which the following norm is finite
		$$|||f|||_r:=\sum_{i=1}^d\sum_{\alpha\in\nn^d}\frac{r^{|\alpha|}}{\alpha!}\left\|\partial^{\alpha}f_i\right\|_{L^{\infty}\left(\overline{B}_{\delta}\right)}.$$
	\end{defi}
	\subsection{An approximate formula for the state, of Magnus-type}\label{black-box}
	This article relies on an approximate representation formula for the state, involving Lie brackets. The goal of the following proposition is to introduce this formula.
	\begin{prop}[Magnus formula] \label{magnus} There exists a unique family of functionals $(\eta_b:\rr^+\times L^1_{\mathrm{loc}}(\rr^+,\rr)^2\to\rr)_{b\in\mathcal{B}}$, called "coordinates of the pseudo-first kind", satisfying the following.
		Let $M\in\nn^*$, $\delta,T>0$, $f_0, f_1, f_2:B(0,2\delta)\to\rr^d$ be analytic vector fields with $f_0(0)=0$ and $T\left\|f_0\right\|_{\infty}\leq\delta$. For $u,v\in L^1((0,T),\rr)$, as $\|(u,v)\|_{L^1} \to 0$,
		\begin{equation} \label{x=Z2+O}
			x(t;(u,v))= \mathcal{Z}_M(t;f,(u,v))(0) +\mathcal{O}\left(  \|(u,v)\|_{L^1}^{M+1}  + \left|x(t;u,v)\right|^{1+\frac{1}{M}}   \right),
		\end{equation}
		where 
		\begin{equation}\label{defzm}\mathcal{Z}_M(t;f,(u,v))(0) =\sum_{b \in \B_{\llbracket 1 ,M \rrbracket}} \eta_b(t,(u,v)) f_b(0).\end{equation}
	\end{prop}
		This proposition stems from \cite[Proposition $161$]{Beauchard_2023}. The ideas of proof are recalled in \cite[Appendix A.$3.2$]{gherdaoui}. The interested reader may find more information on the family $\left(\eta_b:\rr^+\times L^1_{\mathrm{loc}}(\rr^+,\rr)^2\to\rr\right)_{b\in\mathcal{B}}$ in \cite[Proposition $44$]{Beauchard_2023}. 
	\subsection{A black-box estimate}\label{sectionbb}
	Here are a few definitions and notations.
	\begin{defi}[Support]
		Let $\B$ be a Hall set of $\Br(X)$ and $a\in\mathcal{L}(X)$. For $b\in\mathcal{B}$, we denote by $\langle a, b\rangle$ the coefficient of $\ev(b)$ in the expansion of $a$ on the basis $\ev(\mathcal{B})$. We define
		$$\text{supp}(a) := \{b\in\mathcal{B}, \ \langle a,b\rangle\neq 0\}.$$  For $a\in\Br(X)$, $\text{supp}(a):=\text{supp}(\ev(a))$. If $A \subset\Br(X)$, we let supp$(A) := \bigcup_{a\in A}\text{supp}(a)$. 
	\end{defi}
	With this definition, we have, for $a\in\mathcal{L}(X)$, $a=\displaystyle\sum_{b\in\text{supp}(a)}\langle a,b\rangle \ev(b)$.
	\begin{defi}[$\mathcal{F}$]\label{deff}
		Given $q\geq 2$ and $b_1,\cdots,b_q\in\Br(X)$, we define $\mathcal{F}(b_1,\cdots,b_q)$ as the subset of $\Br(X)$ of brackets of $b_1,\cdots, b_q$ involving each of these elements exactly once.
	\end{defi}
	\begin{exs}For $q=2$ and $b_1,b_2\in\Br(X)$, one has $\mathcal{F}(b_1,b_2)=\{(b_1, b_2), (b_2, b_1)\}.$
	For $q=3$ and $b_1,b_2,b_3\in\Br(X)$, one has 
	$$\mathcal{F}(b_1,b_2,b_3)=\left\lbrace((b_{\sigma(1)}, b_{\sigma(2)}),b_{\sigma(3)}), (b_{\sigma(1)}, (b_{\sigma(2)},b_{\sigma(3)})); \ \sigma\in\mathfrak{S}_3\right\rbrace.$$\end{exs}
	To put the strategy described in Section \ref{heuristicstrat} into practice  \textit{i.e.}\ to extract the dominant terms from $\mathcal{Z}_M(t;f,(u,v))(0)$, we will use the following propositions. They guarantee the convergence of the series involved and legitimize the heuristic. This part of the article is based on \cite[Section $4.4$]{beauchard2024unified}. The following proposition allows us to estimate the second term in the right-hand side of \eqref{magnus-heuristique2}.
	\begin{prop}[Estimate of main terms]\label{estim-mainterm} Let $M,L\in\nn^*$. Let $\mathcal{E}\subset\mathcal{B}_{\llbracket 1,M\rrbracket}$. Assume that there
		exist $c>0$ and a functional $\Xi : \rr_+^*\times L^1_{\mathrm{loc}}(\rr^+,\rr)^2\to\rr^+$ such that the following holds:
		for all $b\in\mathcal{E}$, there exists an exponent $\sigma\leq\min(L, |b|)$, such that, for all $t > 0$ and $u,v\in L^1((0,t),\rr)$,
		\begin{equation}\label{bbest1}\left|\xi_b(t,(u,v))\right|\leq \frac{(ct)^{|b|}}{|b|!}t^{-\sigma}\Xi(t,(u,v)).\end{equation}
		Let $\delta,r > 0$ and $f_0,f_1,f_2\in\mathcal{C}^{\omega,r}\left(\overline{B}_{\delta},\rr^d\right)$ be analytic vector fields. Then, for any $r'\in[r/e,r)$, as $(t,\left\|(u,v)\right\|_{L^1})\to 0$,
		$$\sum_{b\in\mathcal{E}}|||\xi_b(t,(u,v))f_b|||_{r'}=\mathcal{O}\left(\Xi(t,(u,v))\right).$$
	\end{prop}
	The following proposition allows us to estimate the third term in the right-hand side of \eqref{magnus-heuristique2}.
	\begin{prop}[Estimate of cross terms]\label{estim-crossterm} Let $M,L\in•\nn^*$. Let $\mathcal{E}\subset\mathcal{B}_{\llbracket 1,M\rrbracket}$. Assume that there exist $c > 0$ and a functional $\Xi: \rr^*_+\times L^1_{\mathrm{loc}}(\rr^+,\rr)^2\to\rr^+$ with $\Xi(t,(u,v))=\mathcal{O}(1)$ such that the following holds:
		for all $q\geq 2$, $b_1\geq\cdots\geq b_q\in\mathcal{B}\setminus\left\lbrace X_0\right\rbrace$ such that $\text{supp}\mathcal{F}(b_1,\cdots,b_q)\cap\mathcal{E}\neq\emptyset$, there exist $\sigma_1,\cdots,\sigma_q\leq L$ with $\sigma_i\leq|b_i|$ and $(\alpha_1,\cdots,\alpha_q)\in[0,1]^q$ with $\alpha_1+\cdots+\alpha_q\geq 1$ such that,  for all $t > 0$ and $u,v\in L^1((0,t),\rr)$,
		\begin{equation}\label{bbest2}\left|\xi_{b_i}(t,(u,v))\right|\leq\frac{(ct)^{|b_i|}}{|b_i|!}t^{-\sigma_i}\left(\Xi(t,(u,v))\right)^{\alpha_i}.\end{equation}
		Let $\delta,r > 0$ and $f_0,f_1,f_2\in\mathcal{C}^{\omega,r}\left(\overline{{B}}_{\delta},\rr^d\right)$ be analytic vector fields. Then, for any $r'\in[r/e,r)$, as $(t,\left\|(u,v)\right\|_{L^1})\to 0$,
		$$\sum_{b\in\mathcal{E}}|||\left(\eta_b-\xi_b\right)(t,(u,v))f_b|||_{r'}=\mathcal{O}\left(\Xi(t,(u,v))\right).$$
	\end{prop}
	The two previous propositions are proved in \cite[Appendix A.$5$]{beauchard2024unified}. The following corollary is a direct consequence of \eqref{defzm} and the two previous propositions. This clarifies several steps of the heuristic -- see Section \ref{heuristicstrat}.
	\begin{crl}\label{bb}
		Let $M, L,r\in\nn^*$. Let $\mathfrak{b}_1,\cdots,\mathfrak{b}_r\in\mathcal{B}_{\llbracket 1,M\rrbracket}$ and $\mathcal{N}\subset\mathcal{B}_{\llbracket 1,M\rrbracket}$. Assume that there exist $c > 0$ and a functional $\Xi : \rr_+^*\times L^1_{\mathrm{loc}}(\rr^+,\rr)^2\to\rr^+$ with $\Xi(t,(u,v)) =\mathcal{O}(1)$ such that
		\begin{enumerate}
			\item the assumptions of Proposition \ref{estim-mainterm} hold for $\mathcal{E}=\mathcal{B}_{\llbracket 1,M\rrbracket }\setminus\left(\mathcal{N}\cup\{\mathfrak{b}_1,\cdots,\mathfrak{b}_r\}\right)$, 
			\item the assumptions of Proposition \ref{estim-crossterm} hold for $\mathcal{E}=\mathcal{B}_{\llbracket 1,M\rrbracket }\setminus\mathcal{N}$.
		\end{enumerate}
		Let $f_0,f_1,f_2$ be analytic vector fields over $\rr^d$. If $\mathcal{P}$ is linear form such that $\restriction{\mathcal{P}}{\mathcal{N}(f)(0)}\equiv 0$, as $(t,\left\|(u,v)\right\|_{L^1})\to 0$,
		\begin{equation}\label{bbconclusion}\mathcal{P}\mathcal{Z}_M(t;f,(u,v))(0)=\sum_{i=1}^r\xi_{\mathfrak{b_i}}(t,(u,v))\mathcal{P}\left(f_{\mathfrak{b}_i}(0)\right)+\mathcal{O}\left(\Xi(t,(u,v))\right).\end{equation}
	\end{crl}
	\subsection{Interpolation inequalities}
	We recall the Gagliardo--Nirenberg interpolation inequalities used in this article and proved in \cite{zbMATH03318089,ASNSP_1959_3_13_2_115_0}.
	\begin{prop}[Gagliardo--Nirenberg inequalities]
		Let $P,q,r,s\in[1,+\infty]$, $ j<l\in\nn$ and $\alpha\in (0,1)$ be such that 
		$$\frac{j}{l}\leq \alpha\quad{and }\quad \frac{1}{P}=j+\left(\frac 1r-l\right)\alpha+\frac{1-\alpha}{q}.$$
		There exists $C>0$ such that, for every $t>0$ and $\ph\in\mathcal{C}^{\infty}([0,t],\rr)$, 
		\begin{equation}\label{gn}\left\|D^j\ph\right\|_{L^P}\leq C\left\| D^l\ph\right\|_{L^r}^{\alpha}\left\|\ph\right\|_{L^q}^{1-\alpha}+Ct^{\frac{1}{P}-j-\frac 1s}\left\|\ph\right\|_{L^s}.
		\end{equation}
	\end{prop}
	\section{Necessary conditions for STLC in the symmetrical case}	
	This section is dedicated to the proof of Theorem \ref{theoremdriftbis}.
	\subsection{Dominant part of the logarithm}
	The expression $|(u_1,\cdots,u_k,v_1,\cdots,v_k)(t)|$ denotes an arbitrary norm of the vector considered in $\rr^{2k}$. We recall that $\pi(k,m)$ is given by \eqref{defpi}. We use Corollary \ref{bb} to extract the main terms from the dynamics. This is the goal of the following statement.
	\begin{lm}\label{dominant} Let $k,m\in\nn^*$. Let $\mathcal{P}$ be a linear form satisfying $\restriction{\mathcal{P}}{\mathcal{N}^m_k(f)(0)}\equiv0$. Then, as $(t,\left\|(u,v)\right\|_{L^1})\to0$,
		\begin{equation}\begin{gathered}\label{dvlplog}\mathcal{P}\mathcal{Z}_{\pi(k,m)}(t;f,(u,v))(0)=\mathcal{P}\left(f_{W^1_k}(0)\right)\xi_{W_k^1}(t,(u,v))+\mathcal{P}\left(f_{W^2_k}(0)\right)\xi_{W_k^2}(t,(u,v))\\+\mathcal{P}\left(f_{C_{2k-1}}(0)\right)\xi_{C_{2k-1}}(t,(u,v))+\mathcal{O}\left(t\left\|(u_k,v_k)\right\|_{L^2}^2+|(u_1,\cdots,u_k,v_1,\cdots,v_k)(t)|^2\right).\end{gathered}\end{equation}
	\end{lm}
	\begin{proof} We apply Corollary \ref{bb} with $M=\pi(k,m)$, $\mathcal{N}=\mathcal{N}_k^m$, $L=2k+2$, $r=3$,  $\sigma=2k+2$, $\mathfrak{b}_1=W_k^1$, $\mathfrak{b}_2=W_k^2$ and $\mathfrak{b}_3=C_{2k-1}$. 
		\begin{enumerate}
			\item \textit{Estimates on the main terms:} let $b\in\mathcal{B}_{\llbracket 1,\pi(k,m)\rrbracket}\setminus\left(\mathcal{N}^{m}_k\cup\left\lbrace\mathfrak{b}_1,\mathfrak{b}_2,\mathfrak{b}_3\right\rbrace\right)$. Then, $n(b)=2$.
			\begin{enumerate}
				\item If $b\in\mathcal{B}_{2,bad}$, $b=W_{j,l}^1$ or $b=W_{j,l}^2$ with $j>k$ or $\left(j=k\text{ and }l\geq 1\right)$. Consequently, $|b|\geq 2k+2$ and the estimates \eqref{estbase-s2bad1} and \eqref{estbase-s2bad2} with $p=1$ and $j_0=k$ give \eqref{bbest1} with $\Xi(t,(u,v))=t\left\|(u_k,v_k)\right\|_{L^2}^2$.
				\item If $b\in\mathcal{B}_{2,good}$, then $b=C_{j,l}$ with $j> 2k-1$ or $\left(j=2k-1\text{ and }l\geq 1\right)$. Similarly, $|b|\geq 2k+2$ and the estimate \eqref{estbase-s2good1} with $p=q=2$ and $k'=k$ gives \eqref{bbest1} with $\Xi(t,(u,v))=t\left\|(u_k,v_k)\right\|_{L^2}^2$ as $k-2-\lfloor\frac{j}{2}\rfloor <0$.
			\end{enumerate}
			\item \textit{Estimates of cross terms: }let $b_1\geq\cdots\geq b_q\in\mathcal{B}\setminus\left\lbrace X_0\right\rbrace$ be such that $n(b_1)+\cdots+n(b_q)\leq \pi(k,m)$ and $\text{supp}\mathcal{F}(b_1,\cdots,b_q)\not\subset\mathcal{N}^m_k$. Let $i\in\llbracket 1,q\rrbracket$.
			\begin{enumerate}
				\item If $b_i=M_j^1$ or $M_j^2$ with $j\in\llbracket 0,k-1\rrbracket$, then by \eqref{egbase-s11} and \eqref{egbase-s12},
				$$|\xi_{b_i}(t,(u,v))|\leq |(u_{j+1},v_{j+1})(t)|.$$
				Then, the estimate \eqref{bbest2} is verified with $\sigma_i=j+1$, $\alpha_i=\frac{1}{2}$ and $\Xi(t,(u,v))=|(u_1,\cdots,u_k,v_1,\cdots,v_k)(t)|^2$.
				\item If $b_i=M_j^1$ or $M_j^2$, with $j\geq k$, $|b_i|\geq k+1$ and the estimates \eqref{estbase-s11} and \eqref{estbase-s12} with $j_0=k$ and $p=2$ give
				$$|\xi_{b_i}(t,(u,v))|\leq\frac{(ct)^{|b_i|}}{|b_i|!}t^{-(k+1)}t^{\frac 12}\left\|(u_k,v_k)\right\|_{L^2}=\frac{(ct)^{|b_i|}}{|b_i|!}t^{-(k+1)}\left(t\left\|(u_k,v_k)\right\|_{L^2}^2\right)^{\frac{1}{2}}.$$
				We obtain \eqref{bbest2} with $\sigma_i=k+1$ and $\alpha_i=\frac{1}{2}$. 
			\end{enumerate}
			Since $\text{supp}\mathcal{F}(b_1,\cdots,b_q)\not\subset\mathcal{N}^m_k$, we have $q=2$ and $b_1,b_2\in\mathcal{B}_1$. Then, as $\alpha_1+\alpha_2=1$, we can apply Corollary \ref{bb} and we obtain the desired equality.
		\end{enumerate}
	\end{proof}
	\subsection{Vectorial relations}
	The purpose of this section is to prove that the condition (BC) implies algebraic properties on the Lie brackets. Using this fact, we will be able to estimate one by one -- in the next paragraph -- the terms $|(u_1,\cdots,u_k,v_1,\cdots,v_k)(t)|$ which appear in the previous proposition.
	\begin{lm}[A bracket relation]\label{bracketdev} Let $k,m\in\nn^*$.
		For all $l\in\llbracket 0,k-1\rrbracket$, for all $(\alpha_{j,1})_{j\in\llbracket 0,l\rrbracket}\in\rr^{l+1},(\alpha_{j,2})_{j\in\llbracket 0,l\rrbracket}\in\rr^{l+1}$, we consider the bracket 
		$$B:=\displaystyle\sum_{j=0}^{l}\alpha_{j,1}M_j^1+\sum_{j=0}^l\alpha_{j,2}M_j^2.$$
		Then, the following expansion holds $$\left[B0^{k-l-1},B0^{k-l}\right]\in\alpha_{l,1}^2W_k^1+\alpha_{l,2}^2W_{k}^2+2\alpha_{l,1}\alpha_{l,2}C_{2k-1}+\mathcal{N}_k^m.$$
	\end{lm}
	This lemma is proved in Appendix \ref{appendicebracket}.
	\begin{lm}\label{contr} Let $e_1,e_2,e_3\in\rr^d$ be three vectors and $N\subset\rr^d$ a vector subspace. If $e_1$, $e_2$, $e_3$, $N$ satisfy (BC), then, there does not exist $(a,b)\in\rr^2\setminus\{(0,0)\}$ such that 
		\begin{equation}\label{dep}a^2e_1+b^2e_2\pm2abe_3\in N.\end{equation}
	\end{lm}
	\begin{proof}By contradiction, assume that there exists $(a,b)\in\rr^2\setminus\{(0,0)\}$ satisfying \eqref{dep}. If $a=0$, then \eqref{dep} gives
		$b^2e_2\in N.$ As, $b\neq 0$ , we obtain $e_2\in N$. This is a contradiction with (BC). Thus, $a\neq 0$.  Similarly, $b\neq 0$. Hence, using $\mathbb{P}$ given by (BC), we have
		\begin{equation}\label{vectorial}
			a^2\mathbb{P}\left(e_1\right)+b^2\mathbb{P}\left(e_2\right)\pm2ab\mathbb{P}\left(e_3\right)=0.
		\end{equation} Nevertheless, by hypothesis (BC) and Young's inequality,
		$$\left|2ab\mathbb{P}(e_3)\right|<2\left|ab\right|\sqrt{\mathbb{P}\left(e_1\right)\mathbb{P}\left(e_2\right)}\leq a^2\mathbb{P}\left(e_1\right)+b^2\mathbb{P}\left(e_2\right).$$
		This is a contradiction with \eqref{vectorial}.
	\end{proof}
	
	\begin{lm}\label{libre}
		Let $k,m\in\nn^*$ and $\nu(k,m):=\left\lfloor\frac{\pi(k,m)}{2}\right\rfloor$. Assume that \eqref{hypo} is verified. 
		\begin{enumerate}
			\item The family 
			$\left(f_{M_0^1}(0),\cdots f_{M_{k-1}^1}(0),f_{M_0^2}(0),\cdots,f_{M_{k-1}^2}(0)\right)$ is linearly independent.
			\item If $\nu(k,m)\geq 2$, $$\spn\left(f_{M_0^1}(0),\cdots,f_{M_{k-1}^1}(0),f_{M_0^2}(0),\cdots,f_{M_{k-1}^2}(0)\right)\cap{S}_{\llbracket 2,\nu(k,m)\rrbracket}(f)(0)=\left\lbrace 0\right\rbrace.$$
		\end{enumerate}
			In particular, $f_1(0)\neq 0$ and $f_2(0)\neq 0$.
	\end{lm}

	\begin{proof}We prove the second point: assume by contradiction that there exist  $(\alpha_{j,1})_{j\in\llbracket 0,k-1\rrbracket},$ $(\alpha_{j,2})_{j\in\llbracket 0,k-1\rrbracket}\in\rr^k$ not all zero and $B\in{S}_{\llbracket 2,\nu(k,m)\rrbracket}(X)$ such that $f_{B_1}(0)=0$, with 
		$$B_1:=\sum_{j=0}^{k-1}\left(\alpha_{j,1}M_j^1+\alpha_{j,2}M_j^2\right)+B.$$ Let $K=\max\{j\in\llbracket 0,k-1\rrbracket; \ (\alpha_{j,1},\alpha_{j,2})\neq (0,0)\}$. As $f_0(0)=0$, $f_{B_2}(0)=0$, with 
		$$B_2:=[B_10^{k-1-K},B_10^{k-K}]\in\alpha_{K,1}^2W_k^1+\alpha_{K,2}^2W_k^2+2\alpha_{K,1}\alpha_{K,2}C_{2k-1}+\mathcal{N}^m_k+{S}_{\llbracket 3,2\nu(k,m)\rrbracket}(X),$$ the expansion is given by Lemma \ref{bracketdev} with $l=K$.
		As $\pi(k,m)\geq 2\nu(k,m)$ and $\nu(k,m)\geq 2$, one has ${S}_{\llbracket 3,2\nu(k,m)\rrbracket}(X)\subseteq{S}_{\llbracket 1,\pi(k,m)\rrbracket\setminus\left\lbrace 2\right\rbrace}(X)\subseteq\mathcal{N}^m_k$. Thus,
		$$\alpha_{K,1}^2f_{W_k^1}(0)+\alpha_{K,2}^2f_{W_k^2}(0)+2\alpha_{K,1}\alpha_{K,2}f_{C_{2k-1}}(0)\in\mathcal{N}^m_k(f)(0).$$  We use Lemma \ref{contr} with $e_1=f_{W_k^1}(0)$, $e_2=f_{W_k^2}(0)$, $e_3=f_{C_{2k-1}}(0)$ and $N=\mathcal{N}_k^m(f)(0)$ to obtain a contradiction. We obtain the first point in the same way, taking $B=0$.
	\end{proof}
	\subsection{Closed-loop estimate}
	Using the algebraic properties proved in the previous section, we can now estimate the terms $\left|(u_1,\cdots,u_k,v_1,\cdots,v_k)(t)\right|$, using the representation formula of the state of Magnus-type -- see Proposition \ref{magnus}. This method is already used by Beauchard and Marbach in \cite{beauchard2024unified}. The elements $u_{i+1}(t)=\xi_{M_1^i}(t,(u,v))$  are part of the dynamics and this fact is used to estimate them, as in \eqref{art2jouetex} for the sixth term.
	\begin{lm}\label{cloop}
		Let $k,m\in\nn^*$ and $\nu(k,m):=\left\lfloor\frac{\pi(k,m)}{2}\right\rfloor$. Assume that \eqref{hypo} holds. Then, as $(t,\left\|(u,v)\right\|_{L^1})\to0$,
		\begin{equation}\label{closed-lopp}|\left(u_1,\cdots,u_k,v_1,\cdots,v_k\right)(t)|=\mathcal{O}\left(t^{\frac 12}\left\|(u_k,v_k)\right\|_{L^2}+\left\|(u,v)\right\|_{L^1}^{\nu(k,m)+1}+\left|x(t;(u,v))\right|\right).\end{equation}
	\end{lm}
	\begin{proof} Let $i\in\llbracket 0,k-1\rrbracket$, By Lemma \ref{libre}, we can consider a linear form $\mathcal{P}$ such that $\restriction{\mathcal{P}}{{\mathcal{N}(f)(0)}}\equiv0$ with $\mathcal{N}:=\mathcal{B}_{\llbracket 2,\nu(k,m)\rrbracket}\cup\{M_j^l; \ j\in\llbracket 0,k-1\rrbracket, l\in\llbracket1,2\rrbracket\}\setminus\{M_i^1\}$ and $\mathcal{P}(f_{M_i^1}(0))=1$. Now, we use Corollary \ref{bb} with $M=\nu(k,m)$, $L=k+1$, $r=1$ and $\mathfrak{b}_1=M_i^1$. 
			\begin{enumerate}
			\item \textit{Estimates on the main terms:} for all $b\in\mathcal{B}_{\llbracket 1,\nu(k,m)\rrbracket}$ such that $b\notin\mathcal{N}_k^m\cup\{\mathfrak{b}_1\}$, we have $n(b)=1$ so $b=M_j^i$ for $j\geq k$, $i\in\llbracket1,2\rrbracket$ and $|b|\geq k+1$. Thus, estimates \eqref{estbase-s11} and \eqref{estbase-s12} with $j_0=k$ and $p=2$ give
			$$|\xi_b(t,(u,v))|\leq\frac{(ct)^{|b|}}{|b|!}t^{-(k+1)}\left(t^{\frac 12}\left\|(u_k,v_k)\right\|_{L^2}\right).$$
			Then, \eqref{bbest1} holds with $\Xi(t,(u,v))=t^{\frac 12}\left\|(u_k,v_k)\right\|_{L^2}$ and $\sigma =k+1$.
			\item \textit{Estimates of cross terms:} for all $b\in\mathcal{B}_{\llbracket 1,\nu(k,m)\rrbracket}\setminus\mathcal{N}_k^m$, we have $n(b)=1$ and there is no cross terms. 
			\end{enumerate}
		Then, Corollary \ref{bb} leads to the equality
		$$\mathcal{P}\mathcal{Z}_{\nu(k,m)}(t;f,(u,v))(0)=u_{i+1}(t)+\mathcal{O}\left(t^{\frac 12}\left\|(u_k,v_k)\right\|_{L^2}\right).$$
		Using the Magnus formula given by Proposition \ref{magnus} with $M=\nu(k,m)$, we finally get
		$$\mathcal{P}x(t;(u,v))=u_{i+1}(t)+\mathcal{O}\left(t^{\frac 12} \left\|(u_k,v_k)\right\|_{L^2}+\left\|(u,v)\right\|_{L^1}^{\nu(k,m)+1}+\left|x(t;(u,v))\right|^{1+\frac{1}{\nu(k,m)}}\right).$$
		We obtain the result. We can obtain the same estimate for $|v_{i+1}(t)|$, $i\in\llbracket 0,k-1\rrbracket$.
	\end{proof}
	\subsection{Interpolation inequality}
	The representation formula of the state -- see Proposition \ref{magnus} with $M=\pi(k,m)$ -- makes a strong link between $x(t;(u,v))$ and $\mathcal{Z}_{\pi(k,m)}(t;f,(u,v))(0)$.  Lemma \ref{dominant} gives an expansion of $\mathbb{P}\mathcal{Z}_{\pi(k,m)}(t;f,(u,v))(0)$. Furthermore, the edge terms $|(u_1,\cdots,u_k,v_1,\cdots,v_k)(t)|$ are estimated by Lemma \ref{cloop}. However, there is an error term in the Magnus-type formula shaped as $\mathcal{O}\left(\left\|(u,v)\right\|_{L^1}^{\pi(k,m)+1}\right)$. We then relate this quantity to the size of the drift $\left\|(u_k,v_k)\right\|_{L^2}^2$, thanks to the Gagliardo--Nirenberg interpolation inequalities. This is the purpose of the following lemma.
	\begin{lm}\label{gna}
		Let $k,m\in\nn^*$ and $p\in[1,+\infty]$. There exists $C>0$ such that, for every $t>0$ and $u\in W^{m,p}((0,t),\rr)$,
		\begin{equation}\label{gn-m}\left\|u\right\|_{L^1}^{\pi(k,m)+1}\leq  C t^{\pi(k,m)-2k}\left\|u\right\|_{W^{m,p}}^{\pi(k,m)-1}\left\|u_k\right\|_{L^2}^2.
		\end{equation}
	\end{lm}
	\begin{proof} For simplicity, we write $\pi$ instead of $\pi(k,m)$.
		We use the Gagliardo--Nirenberg interpolation inequalities \eqref{gn} with $P =\frac{2(m+k)p}{2k+mp}$, $q=2$, $r=p$, $s=2$, $j=k$, $l=m+k$,  $\alpha=\frac{k}{k+m}$ and $\varphi=u_k$ and we obtain
		\begin{equation}\label{gn1}\left\| u\right\|_{L^P}\leq C\left\|u^{(m)}\right\|^{\alpha}_{L^p}\left\|u_k\right\|_{L^2}^{1-\alpha}+ Ct^{\frac{1}{P}-(k+\frac 12)}\left\|u_k\right\|_{L^2}.
		\end{equation}
		Moreover, using Hölder's inequality,
		\begin{equation}\label{gn2}\left\|u\right\|_{L^1}\leq t^{1-\frac{1}{P}}\left\|u\right\|_{L^P}.\end{equation}
		Using \eqref{gn1} and \eqref{gn2}, we obtain
		$$\left\|u\right\|_{L^1}^{\pi+1}\leq C t^{(\pi+1)\left(1-\frac{1}{P}\right)}\left(\left\|u^{(m)}\right\|^{\alpha(\pi+1)}_{L^p}\left\|u_k\right\|_{L^2}^{(1-\alpha)(\pi+1)}+ t^{(\pi+1)\left(\frac{1}{P}-(k+\frac 12)\right)}\left\|u_k\right\|_{L^2}^{\pi+1}\right)$$
		Thus, if we define $\beta:=1+\frac{2k}{m}$, then, $(1-\alpha)(1+\beta)=2$. We get
		\begin{equation}\label{gn3}\left\|u\right\|_{L^1}^{\pi+1}\leq C t^{(\pi+1)\left(1-\frac{1}{P}\right)}\left(\left\|u\right\|_{W^{m,p}}^{\alpha(\pi+1)}\left\|u_k\right\|_{L^2}^{(1-\alpha)\left(\pi-\beta\right)}+ t^{(\pi+1)\left(\frac{1}{P}-(k+\frac 12)\right)}\left\|u_k\right\|_{L^2}^{\pi-1}\right)\left\|u_k\right\|_{L^2}^2.\end{equation}
		Moreover, 
		\begin{equation}\label{gn4}\left\|u_k\right\|_{L^2}\leq t^{k+\frac 12}\left\|u\right\|_{L^{\infty}}\leq Ct^{k+\frac 12}\left\|u\right\|_{W^{m,p}}.\end{equation}
		Using \eqref{gn4} in \eqref{gn3}, we obtain
		\begin{equation*}\begin{gathered}\left\|u\right\|_{L^1}^{\pi+1}\leq C t^{(\pi+1)\left(1-\frac{1}{P}\right)}\left(\left\|u\right\|_{W^{m,p}}^{\alpha(\pi+1)+(1-\alpha)\left(\pi-\beta\right)}t^{(k+\frac 12)(1-\alpha)\left(\pi-\beta\right)}\right.\\\left. +t^{(\pi+1)\left(\frac{1}{P}-(k+\frac 12)\right)+(\pi-1)\left(k+\frac 12\right)}\left\|u\right\|_{W^{m,p}}^{\pi-1}\right)\left\|u_k\right\|_{L^2}^2.\end{gathered}\end{equation*}
		As $\alpha(\pi+1)+(1-\alpha)\left(\pi-\beta\right)=\pi-1$, we obtain
		$$\left\|u\right\|_{L^1}^{\pi+1}\leq C t^{(\pi+1)\left(1-\frac{1}{P}\right)}\left(t^{(k+\frac 12)(1-\alpha)\left(\pi+1\right)-(2k+1)}+t^{\frac{\pi+1}{P}-(2k+1)}\right)\left\|u\right\|_{W^{m,p}}^{\pi-1}\left\|u_k\right\|_{L^2}^2.$$
		Finally, 
		$$\left(k+\frac 12\right)(1-\alpha)=\frac{(2k+1)m}{2(k+m)}\geq \frac{1}{P},$$
		Thus,
		$$\left\|u\right\|_{L^1}^{\pi+1}\leq C t^{(\pi+1)\left(1-\frac{1}{P}\right)}t^{\frac{\pi+1}{P}-(2k+1)}\left\|u\right\|_{W^{m,p}}^{\pi-1}\left\|u_k\right\|_{L^2}^2.$$
	\end{proof}
	\subsection{Proof of the drift}
	We can now use the Magnus-type representation formula given by Proposition \ref{magnus}, the expansion of $\mathbb{P}\mathcal{Z}_{\pi(k,m)}(t,f,(u,v))(0)$ given by Lemma \ref{dominant}, the estimate of Lemma \ref{cloop} and the interpolation inequality given by Lemma \ref{gna} to prove Theorem \ref{theoremdriftbis}.
	\begin{proof}[Proof of Theorem \ref{theoremdriftbis}] Let $k,m\in\nn^*$ and $p\in[1,+\infty]$. We will write $\pi$ instead of $\pi(k,m)$. Let $e_1:=f_{W_k^1}(0)$, $e_2:=f_{W_k^2}(0)$ and $e_3:=f_{C_{2k-1}}(0)$. Let $\mathbb{P}$ be a linear form given by (BC). The Magnus-type expansion formula given by Proposition \ref{magnus} with $M=\pi$, the equalities \eqref{egbase-s2bad1}, \eqref{egbase-s2bad2} and \eqref{egbase-s2good} and \eqref{dvlplog} give, as $(t,\left\|(u,v)\right\|_{L^1})\to0$,
		\begin{equation}\begin{gathered}\label{fin1}\mathbb{P}x(t;(u,v))=\int_0^t\left(\mathbb{P}(e_1)\frac{u_k^2}{2}+\mathbb{P}(e_2)\frac{v_k^2}{2}+\mathbb{P}(e_3)u_kv_k\right)+\mathcal{O}\left(t\left\|(u_k,v_k)\right\|_{L^2}^2\right.\\\left.+|(u_1,\cdots,u_k,v_1,\cdots,v_k)(t)|^2+\left\|(u,v)\right\|^{\pi+1}_{L^1}+\left|x(t;(u,v))\right|^{1+\frac{1}{\pi}}\right).\end{gathered}\end{equation}
		The closed-loop estimates \eqref{closed-lopp} gives, with $\nu:=\left\lfloor\frac{\pi}{2}\right\rfloor$, 
		\begin{equation}\label{ccl}
			|(u_1,\cdots,u_k,v_1,\cdots,v_k)(t)|^2=\mathcal{O}\left(t\left\|(u_k,v_k)\right\|_{L^2}^2+\left\|(u,v)\right\|_{L^1}^{2\nu+2}+\left|x(t;(u,v))\right|^2\right).
		\end{equation}
		By definition of $\nu$, we have $2(\nu+1)\geq\pi+1$. In particular, as $\left\|(u,v)\right\|_{L^1}\to0$,
		$$\left\|(u,v)\right\|_{L^1}^{2\nu+2}=\mathcal{O}\left(\left\|(u,v)\right\|_{L^1}^{\pi+1}\right).$$
		Using \eqref{ccl} in \eqref{fin1} and the interpolation inequality \eqref{gn-m}, we get
		\begin{equation*}\begin{gathered}\mathbb{P}x(t;(u,v))=\int_0^t\left(\mathbb{P}(e_1)\frac{u_k^2}{2}+\mathbb{P}(e_2)\frac{v_k^2}{2}+\mathbb{P}(e_3)u_kv_k\right)\\+\mathcal{O}\left(\left(t+t^{\pi-2k}\left\|(u,v)\right\|_{W^{m,p}}^{\pi-1}\right)\left\|(u_k,v_k)\right\|_{L^2}^2+\left|x(t;(u,v))\right|^{1+\frac{1}{\pi}}\right).\end{gathered}\end{equation*}
		We prove that the system \eqref{affine-syst} has a drift in the regime $(t,t^{\alpha}\left\|(u,v)\right\|_{W^{m,p}})\to 0$, with $\alpha=\frac{\pi-2k}{\pi-1}$: by definition, there exist $C,\rho>0$ such that, for every $t\in (0,\rho)$, there exists $\eta>0$ \textit{s.t.}\ for every $u,v\in W^{m,p}((0,t),\rr)$ with $\left\|(u,v)\right\|_{W^{m,p}}\leq\eta$,
		\begin{equation}\begin{gathered}\label{reste}\left|\mathbb{P}x(t;(u,v))-\int_0^t\left(\mathbb{P}(e_1)\frac{u_k^2}{2}+\mathbb{P}(e_2)\frac{v_k^2}{2}+\mathbb{P}(e_3)u_kv_k\right)\right|\\\leq C\left(\left(t+t^{\pi-2k}\left\|(u,v)\right\|_{W^{m,p}}^{\pi-1}\right)\left\|(u_k,v_k)\right\|_{L^2}^2+\left|x(t;(u,v))\right|^{1+\frac{1}{\pi}}\right).\end{gathered}\end{equation}
		Let $\gamma:=\frac{\left|\mathbb{P}(e_3)\right|}{\sqrt{\mathbb{P}(e_1)\mathbb{P}(e_2)}}<1$, by hypothesis (BC). Using Young's inequality, we obtain
		\begin{equation}\label{coercive}
			\int_0^t\left(\mathbb{P}(e_1)\frac{u_k^2}{2}+\mathbb{P}(e_2)\frac{v_k^2}{2}+\mathbb{P}(e_3)u_kv_k\right)\geq K\int_0^t\left(u_k^2+v_k^2\right),\end{equation}
		with $K:=\displaystyle\frac{1}{2}\left(1-\gamma\right)\min\left(\mathbb{P}(e_1),\mathbb{P}(e_2)\right)$.
		Thus, for all $t\in\left(0,\min\left(\rho,\frac{K}{4C}\right)\right)$, for all $u,v\in W^{m,p}((0,t),\rr)$, with $\left\|(u,v)\right\|_{W^{m,p}}\leq\min\left(\eta,\left(t^{2k-\pi}\frac{K}{4C}\right)^{\frac{1}{\pi-1}}\right)$, the equalities \eqref{reste} and \eqref{coercive} lead to
		$$\mathbb{P}x(t;(u,v))\geq \frac{K}{2}\Delta(u,v)-C\left|x(t;(u,v))\right|^{1+\frac{1}{\pi}},$$
		with $\Delta:(u,v)\in L^1_{\mathrm{loc}}(\rr^+,\rr)^2\mapsto \displaystyle\int_0^t\left(u_k^2+v_k^2\right)\in\rr^+$. Then, the system \eqref{affine-syst} has a drift along $e_1+e_2$ parallel to $\mathcal{N}^m_k(f)(0)$ with strength $\Delta$ as $\left(t,t^{\alpha}\left\|(u,v)\right\|_{W^{m,p}}\right)\to 0$. This concludes the proof of Theorem \ref{theoremdriftbis}.
	\end{proof}
	\section{Necessary conditions for STLC in the asymmetrical case}\label{sectionassym}
	This section is devoted to the study of systems with asymmetrical drift, as introduced in Example \ref{exassymgeneral}.
	\subsection{Main theorem}
	For $A_1,A_2\subset \nn$, we recall that the sets $S_{A_1}(X)$ and $S_{A_1,A_2}(X)$ are defined in \eqref{art2sa1a2}, \eqref{art2sa}.
		\begin{theorem}\label{theoremdriftbis2} Let $f_0,f_1,f_2$ be analytic vector fields over $\rr^d$ with $f_0(0)=0$. Let $k,k',m,m'\in\nn^*$ be such that $k'\leq k$. We recall that $\pi$ is given by \eqref{defpi}.
		We define \begin{equation}\begin{gathered}\label{defnkk'mm'}\mathcal{N}^{m,m'}_{k,k'}={S}_{\llbracket 1,\pi(k,m)\rrbracket\setminus\left\lbrace 2\right\rbrace}(X)\cap{S}_{\llbracket 0,\pi(k,m)\rrbracket,\llbracket 0,\pi(k',m')\rrbracket}(X)\\\cup\left\lbrace C_{j,l}; \ j\in\llbracket 0,k+k'-2\rrbracket,l\in\nn\right\rbrace\cup\left\lbrace W_{i,l}^1, W_{j,l}^2; \ (i,j)\in\llbracket 1,k-1\rrbracket\times\llbracket 1,k'-1\rrbracket,l\in\nn\right\rbrace,\end{gathered}\end{equation}
		where the last set in the right-hand side is empty if $k=1$ or $k'=1$. We assume that
		\begin{equation}\label{hypo2}f_{W_k^1}(0), \quad f_{W_{k'}^2}(0), \quad f_{C_{k+k'-1}}(0), \quad \mathcal{N}_{k,k'}^{m,m'}(f)(0) \quad\text{ satisfy (BC)}.\end{equation} Then, for all $p,p'\in[1,+\infty]$, the system \eqref{affine-syst} has a drift along $f_{W_k^1}(0)+f_{W_{k'}^2}(0)$ parallel to $\mathcal{N}_{k,k'}^{m,m'}(f)(0)$ with strength $\Delta :(u,v)\in L^1_{\mathrm{loc}}(\rr^+,\rr)^2\mapsto\displaystyle\int_0^t\left(u_k^2+v_{k'}^2\right)\in\rr^+$ as $\left(t,t^{\alpha}\left\|(u,v)\right\|_{W^{m,p}\times W^{m',p'}}\right)\to 0$, where $\alpha:=\frac{\pi(k,m)-2k}{\pi(k,m)-1}$ and $\alpha':=\frac{\pi(k',m')-2k'}{\pi(k',m')-1}$. As a consequence, the system \eqref{affine-syst} is not $W^{m,p}\times W^{m',p'}$-STLC -- see Lemma \ref{lemmedrift}.
	\end{theorem}
	\begin{rmq}
		The case when $k\leq k'$ can be proved in the same way.
	\end{rmq}
	\begin{ex}
		Let us return to Example \ref{exintegrateurintro}. One has
		$$f_{W_k^1}(0)= f_{W_{k'}^2}(0)=2e_{k+k'+1}, \quad f_{C_1}(0)=\alpha e_{k+k'+1},\quad \mathcal{N}^{m,m'}_{k,k'}(f)(0)=\spn\left(e_i\right)_{ i\in\llbracket 1,k+k'\rrbracket},$$
		for any integer $m,m'\in\nn^*$. Thus, the form $\mathbb{P}:(x_i)_{i\in\llbracket1, k+k'+1\rrbracket}\in\rr^{k+k'+1}\mapsto x_{k+k'+1}\in\rr$ ensures that the system is not $W^{m,p}\times W^{m',p'}$-STLC for every $p,p'\in[1,+\infty]$, when $|\alpha|< 2$.
	\end{ex}
	Using the same approach as for Theorem \ref{theoremdriftbis}, Theorem \ref{theoremdriftbis2} can be reformulated in an effective form as follows. Once again, the following statement is directly derived from Theorem \ref{theoremdriftbis2} \textit{via} Proposition \ref{bclink}. One recalls that the set $\mathcal{N}_{k,k'}^{m,m'}$ is defined in \eqref{defnkk'mm'}.
	\begin{theorem}\label{theoremdrift2}  Let $f_0,f_1,f_2$ be analytic vector fields over $\rr^d$ with $f_0(0)=0$.
		Let $k,k',m,m' \in\nn^*$ be such that $k'\leq k$. Let $\sigma:\rr^d\to\rr^d/\mathcal{N}_{k,k'}^{m,m'}(f)(0)$ be the canonical surjection, $\tilde{e}_1:=\sigma\left(f_{W_k^1}(0)\right)$, $\tilde{e}_2:=\sigma\left(f_{W_{k'}^2}(0)\right)$ and $\tilde{e}_3:=\sigma\left(f_{C_{k+k'-1}}(0)\right)$.
		If the system \eqref{affine-syst} is $W^{m,p}\times W^{m',p'}$-STLC for given $p,p'\in[1,+\infty]$, one of the following conditions is satisfied
		\begin{enumerate}
			\item[\textbullet] $\tilde{e}_1=0$ or $\tilde{e}_2=0$,
			\item[\textbullet] $(\tilde{e}_1,\tilde{e}_2)$ is a linearly independent family and $\tilde{e}_3=a\tilde{e}_1+b\tilde{e}_2$ with $ab\geq \frac{1}{4}$,
			\item[\textbullet] $\tilde{e}_2=\beta\tilde{e}_1$ with $\beta<0$,
			\item[\textbullet] $\tilde{e}_2=\beta \tilde{e}_1$, $\tilde{e}_3=\gamma \tilde{e}_1$ with $\beta\leq\gamma^2$ and $\beta\neq 0$.
		\end{enumerate}
	\end{theorem}
	\begin{ex}
		Let us return to Example \ref{exintegrateurintro}. One has $\tilde{e}_1=2$, $\tilde{e}_2=2$ and $\tilde{e}_3=\alpha$. If  $|\alpha|<2$, the four points of Theorem \ref{theoremdrift2} are not satisfied. By contraposition, for every $m,m'\in\nn^*$, $p,p'\in [1,+\infty]$, the system is not $W^{m,p}\times W^{m',p'}$-STLC.
	\end{ex}
	\begin{ex}Theorems \ref{theoremdriftbis2} and \ref{theoremdrift2} can be used to study \eqref{exassym} with $k=2$, $k'=1$, $m,m'\in\nn^*$ and $p,p'\in[1,+\infty]$.\end{ex}
	\subsection{A new truncation in the Magnus-type representation formula}
	In order to prove Theorem \ref{theoremdriftbis2}, we first give an asymmetrical Magnus-type representation formula. This is the purpose of the following statement.
	\begin{prop}[Asymmetrical Magnus expansion]\label{magnus-assym}
		Let $M,N\in\nn^*$ be such that $N\leq M$, let $\delta,T>0$ and $f_0, f_1, f_2:B(0,2\delta)\to\rr^d$ be analytic vector fields  with $f_0(0)=0$ and $T\left\|f_0\right\|_{\infty}\leq\delta$. For $u,v\in L^1((0,T),\rr)$, as $\|(u,v)\|_{L^1} \to 0$,
		\begin{equation}\label{x=e^Z2bis+O}
			x(t;(u,v))=\mathcal{Z}_{M,N}(t;f,(u,v))(0)+ \mathcal{O}\left(\left\|u\right\|_{L^1}^{M+1}+\left\| v\right\|_{L^1}^{N+1}+\left|x(t;(u,v))\right|^{1+\frac 1M}\right),
		\end{equation}
		where 
		\begin{equation}\label{zmn}
			\mathcal{Z}_{M,N}(t;f,(u,v)):=\sum_{\substack{b\in\mathcal{B}_{\llbracket 1,M\rrbracket }\\n_2(b)\leq N}}\eta_b(t,(u,v))f_b.
		\end{equation}
	\end{prop}
	\begin{proof}By definition
		$$\mathcal{Z}_{M}(t;f,(u,v))(0)=\mathcal{Z}_{M,N}(t;f,(u,v))(0)+\sum_{\substack{b\in\mathcal{B}_{\llbracket 1,M\rrbracket},\\n_2(b)\geq N+1}}\eta_b(t,(u,v))f_b(0).$$
		We use analytic estimates, as for Propositions \ref{estim-mainterm} and \ref{estim-crossterm} and the estimate on the coordinates of pseudo-first kind given by \cite[Proposition $52$]{Beauchard_2023} to obtain, as $\left\|(u,v)\right\|_{L^1}\to 0$, 
		$$\sum_{\substack{b\in\mathcal{B}_{\llbracket 1,M\rrbracket},\\n_2(b)\geq N+1}}|||\eta_b(t,(u,v))f_b|||_{r'}=\mathcal{O}\left(\left\|v\right\|_{L^1}^{N+1}\right),$$
		for a given $r'>0$.
		The Magnus-type representation formula \eqref{x=Z2+O} leads to the conclusion.
	\end{proof}
	We are now in a position to prove Theorem \ref{theoremdriftbis2}. Let $k,k',m,m'\in\nn^*$.  From now on, we will sometimes refer to $\pi(k,m)$ as $\pi$, $\pi(k',m')$ as $\pi'$ and $\mathcal{N}_{k,k'}^{m,m'}$ as $\mathcal{N}$.
	
	\subsection{Dominant part of the logarithm}
	\begin{lm}\label{dominant2} Let $k,k',m,m'\in\nn^*$ be such that $k'\leq k$. Let $\mathcal{P}$ be a linear form satisfying $\restriction{\mathcal{P}}{\mathcal{N}(f)(0)}\equiv0$. Then, as $(t,\left\|(u,v)\right\|_{L^1})\to0$,
		\begin{equation}\begin{gathered}\label{dvlplog2}\mathcal{P}\mathcal{Z}_{\pi,\pi'}(t;f,(u,v))(0)=\mathcal{P}\left(f_{W^1_k}(0)\right)\xi_{W_k^1}(t,(u,v))+\mathcal{P}\left(f_{W^2_{k'}}(0)\right)\xi_{W_{k'}^2}(t,(u,v))\\+\mathcal{P}\left(f_{C_{k+k'-1}}(0)\right)\xi_{C_{k+k'-1}}(t,(u,v))+\mathcal{O}\left(t\left\|(u_k,v_{k'})\right\|_{L^2}^2+|(u_1,\cdots,u_k,v_1,\cdots,v_{k'})(t)|^2\right).\end{gathered}\end{equation}
	\end{lm}
	\begin{proof} We fix $M=\pi$,  $\mathcal{N}=\mathcal{N}_{k,k'}^{m,m'}$  and $r=1$. Let $\mathcal{E}:=\mathcal{B}_{\llbracket 1,\pi\rrbracket}\cap\{b\in\Br(X); \ n_2(b)\leq \pi'\}\setminus\left(\mathcal{N}\cup\{W_k^1,W_{k'}^2,C_{k+k'-1}\}\cup\{C_{j,l}; \ l+\lfloor\frac{j}{2}\rfloor\leq k-2\}\right)$ and $N=\pi'$.
		\begin{enumerate}
			\item\textit{Estimates of the main terms:} let $b\in\mathcal{E}$. Then $n(b)=2$.
			\begin{enumerate}
				\item If $b\in \mathcal{E}\cap\{b\in\Br(X); \ n_1(b)=2\}$, $b=W_{j,l}^1$ with $j>k$ or $\left(j=k\text{ and }l\geq 1\right)$. Thus $|b|\geq 2k+2$ and the estimate \eqref{estbase-s2bad1} with $j_0=k$ and $p=1$ gives \eqref{bbest1} with $\Xi(t,(u,v)):=t\left\|u_k\right\|_{L^2}^2$.
				\item If $b\in \mathcal{E}\cap\{b\in\Br(X); \ n_2(b)=2\}$, $b=W_{j,l}^2$ with $j>k'$ or $\left(j=k'\text{ and }l\geq 1\right)$. Thus $|b|\geq 2k'+2$ and the estimate \eqref{estbase-s2bad2} with $j_0=k'$ and $p=1$ gives \eqref{bbest1} with $\Xi(t,(u,v)):=t\left\|v_{k'}\right\|_{L^2}^2$.
				\item If $b\in \mathcal{E}\cap\mathcal{B}_{2,good}$ and $b=C_{j,l}$ with $j\geq k+k'$ or $\left(j\geq k+k'-1\text{ and }l\geq 1\right)$. Then, $|b|\geq k+k'+2$ and the estimate \eqref{estbase-s2good1} with $p=q=2$ gives \eqref{bbest1} with $\Xi(t,(u,v)):=t\left\|(u_k,v_{k'})\right\|_{L^2}^2$.
			\end{enumerate}
			By applying Proposition \ref{estim-mainterm} to each of the three sets in the partition of $\mathcal{E}$ given above (once with $\sigma=L=2k+2$, $\mathfrak{b}_1=W_k^1$, once with $\sigma=L=2k'+2$, $\mathfrak{b}_1=W_{k'}^1$ and once with $\sigma=L=k+k'+2$, $\mathfrak{b}_1=C_{k+k'-1}$), we obtain as $(t,\left\|(u,v)\right\|_{L^1})\to 0$,
			$$\sum_{b\in\mathcal{E}}|||\xi_b(t,(u,v))f_b|||_{r'}=\mathcal{O}\left(t\left\|(u_k,v_{k'})\right\|_{L^2}^2\right),$$
			for a given $r'>0$.
			We finally  examine the brackets in $\{C_{j,l}; \ l+\lfloor\frac{j}{2}\rfloor\leq k-2\}$, where the set is finite. Estimate \eqref{estbase-s2good1} with $p=q=2$ gives, as $(t,\left\|(u,v)\right\|_{L^1})\to 0$,
			$$\sum_{l+\lfloor\frac{j}{2}\rfloor\leq k-2}|||\xi_{C_{j,l}}(t,(u,v))f_{C_{j,l}}|||_{r'}=\mathcal{O}\left(t\left\|(u_k,v_{k'})\right\|_{L^2}^2+|(u_1,\cdots,u_k)(t)|^2\right).$$
			Finally,
			$$\sum_{\substack{b\in\mathcal{B}_{\llbracket 1,\pi\rrbracket},\ n_2(b)\leq\pi', \\ b\notin\mathcal{N}\cup\{\mathfrak{b}_1,\mathfrak{b}_2,\mathfrak{b}_3\}}}|||\xi_b(t,(u,v))f_b|||_{r'}=\mathcal{O}\left(t\left\|(u_k,v_{k'})\right\|_{L^2}^2+\left|(u_1,\cdots,u_k)(t)\right|^2\right).$$
			\item\textit{Estimates of cross terms:} we apply Proposition \ref{estim-crossterm} with the set $\mathcal{E}:=\mathcal{B}_{\llbracket 1,\pi\rrbracket}\cap\{b\in\Br(X); \ n_2(b)\leq \pi'\}\setminus\mathcal{N}$. Let $b_1\geq\cdots\geq b_q\in\mathcal{B}\setminus\left\lbrace X_0\right\rbrace$ be such that $n_1(b_1)+\cdots+n_1(b_q)\leq \pi$, $n_2(b_1)+\cdots+n_2(b_q)\leq \pi'$ and $\text{supp}\mathcal{F}(b_1,\cdots,b_q)\not\subset\mathcal{N}$. Let $i\in\llbracket 1,q\rrbracket$.
			\begin{enumerate}
				\item If $b_i=M_j^1$ with $j\in\llbracket 0,k-1\rrbracket$ or $b_i=M_l^2$ with $l\in\llbracket 0,k'-1\rrbracket$, then by \eqref{egbase-s11} and \eqref{egbase-s12},
				$$|\xi_{b_i}(t,(u,v))|\leq |(u_{j+1},v_{l+1})(t)|.$$
				Then, the estimate \eqref{bbest2} is verified with $\sigma_i=j+1$ or  $\sigma_i=l+1$, $\alpha_i=\frac{1}{2}$ and $\Xi(t,(u,v))=|(u_1,\cdots,u_k,v_1,\cdots,v_k')(t)|^2$.
				\item If $b_i=M_j^1$ with $j\geq k$, $|b_i|\geq k+1$ and the estimate \eqref{estbase-s11} with $j_0=k$ and $p=2$ gives
				$$|\xi_{b_i}(t,(u,v))|\leq\frac{(ct)^{|b_i|}}{|b_i|!}t^{-(k+1)}t^{\frac 12}\left\|u_k\right\|_{L^2}=\frac{(ct)^{|b_i|}}{|b_i|!}t^{-(k+1)}\left(t\left\|u_k\right\|_{L^2}^2\right)^{\frac{1}{2}}.$$
				We obtain \eqref{bbest2} with $\sigma_i=k+1$ and $\alpha_i=\frac{1}{2}$. 
				\item If $b_i=M_l^2$ with $l\geq k'$, $|b_i|\geq k'+1$ and the estimate \eqref{estbase-s12} with $j_0=k'$ and $p=2$ gives
				$$|\xi_{b_i}(t,(u,v))|\leq\frac{(ct)^{|b_i|}}{|b_i|!}t^{-(k'+1)}t^{\frac 12}\left\|v_{k'}\right\|_{L^2}=\frac{(ct)^{|b_i|}}{|b_i|!}t^{-(k'+1)}\left(t\left\|v_{k'}\right\|_{L^2}^2\right)^{\frac{1}{2}}.$$
				We obtain \eqref{bbest2} with $\sigma_i=k'+1$ and $\alpha_i=\frac{1}{2}$. 
			\end{enumerate}
			Since $\text{supp}\mathcal{F}(b_1,\cdots,b_q)\not\subset\mathcal{N}$, we have $q=2$ and $b_1,b_2\in\mathcal{B}_1$. Moreover, $\alpha_1+\alpha_2=1$.
		\end{enumerate}
		Thus, the definition of $\mathcal{Z}_{\pi,\pi'}$ -- see \eqref{zmn} -- leads to the result.
	\end{proof}
	\subsection{Vectorial relations}
	\begin{lm}[A bracket relation]\label{bracketdev2} Let $k,k',m,m'\in\nn^*$ be such that $k'\leq k$.
		For all $l\in\llbracket 0,k'-1\rrbracket$, for all $(\alpha_j)_{j\in\llbracket 0,l+k-k'\rrbracket}\in\rr^{l+k-k'+1},(\beta_j)_{j\in\llbracket 0,l\rrbracket}\in\rr^{l+1}$, we consider the bracket 
		$$B:=\displaystyle\sum_{j=0}^{l+k-k'}\alpha_jM_j^1+\sum_{j=0}^l\beta_jM_j^2.$$
		Then, the following expansion holds $$\left[B0^{k'-l-1},B0^{k'-l}\right]\in\alpha_{l+k-k'}^2W_k^1+\beta_{l}^2W_{k'}^2+2(-1)^{\lfloor\frac{k-k'}{2}\rfloor}\alpha_{l+k-k'}\beta_{l}C_{k+k'-1}+\mathcal{N}_{k,k'}^{m,m'}.$$
	\end{lm}
	This lemma is proved in Appendix \ref{appendicebracket2}. This is a generalization of Lemma \ref{bracketdev}, in the asymmetrical case.
	\begin{lm}\label{libre2}
		Let $k,k',m,m'\in\nn^*$ be such that $k'\leq k$ and $\nu:=\left\lfloor\frac{\pi}{2}\right\rfloor$, $\nu':=\left\lfloor\frac{\pi'}{2}\right\rfloor$. Assume that \eqref{hypo2} is verified. Then, 
		\begin{enumerate}
			\item the family 
			$\left(f_{M_0^1}(0),\cdots f_{M_{k-1}^1}(0),f_{M_0^2}(0),\cdots,f_{M_{k'-1}^2}(0)\right)$ is linearly independent.
			\item if $\nu\geq 2$, \begin{multline*}\spn\left(f_{M_0^1}(0),\cdots,f_{M_{k-1}^1}(0),f_{M_0^2}(0),\cdots,f_{M_{k'-1}^2}(0)\right)\\\cap{S}_{\llbracket 2,\nu\rrbracket}(f)(0)\cap{S}_{\llbracket 0,\nu\rrbracket,\llbracket 0,\nu'\rrbracket}(f)(0)=\left\lbrace 0\right\rbrace.\end{multline*}
		\end{enumerate}
	\end{lm}
	\begin{proof} We prove the result in the same way as Lemma \ref{libre}. Let us start with the second point. Assume by contradiction that there exist $(\alpha_j)_{j\in\llbracket 0,k-1\rrbracket}\in\rr^k,(\beta_j)_{j\in\llbracket 0,k'-1\rrbracket}\in\rr^{k'}$, not all zero and $B\in{S}_{\llbracket 2,\nu\rrbracket}(X)\cap {S}_{\llbracket 0,\nu\rrbracket,\llbracket 0,\nu'\rrbracket}(X)$ such that $f_{B_1}(0)=0$, with 
		$$B_1:=\sum_{j=0}^{k-1}\alpha_jM_j^1+\sum_{j=0}^{k'-1}\beta_jM_j^2+B.$$ 
		\begin{enumerate} 
			\item Firstly, assume that $\alpha_{k-k'}=\cdots=\alpha_{k-1}=\beta_0=\cdots=\beta_{k'-1}=0$. We consider $K:=\max\{j\in\llbracket 0,k-k'-1\rrbracket; \ \alpha_j\neq 0\}.$
			As $f_0(0)=0$, we have $f_{B_2}(0)=0$ with
			$$B_2:=[B_10^{k-1-K},B_10^{k-K}]\in \alpha_K^2W_k^1+\mathcal{N}+{S}_{\llbracket 3,2\nu\rrbracket}\cap {S}_{\llbracket 0,2\nu\rrbracket,\llbracket 0,2\nu'\rrbracket}(X).$$
			Moreover, by definition of $\nu$ and $\nu'$,
			\begin{equation}\label{inclusion}{S}_{\llbracket 3,2\nu\rrbracket}\cap {S}_{\llbracket 0,2\nu\rrbracket,\llbracket 0,2\nu'\rrbracket}(X)\subseteq{S}_{\llbracket 1,\pi\rrbracket\setminus\left\lbrace 2\right\rbrace}\cap {S}_{\llbracket 0,\pi\rrbracket,\llbracket 0,\pi'\rrbracket}(X)\subseteq\mathcal{N}.\end{equation}
			As $\alpha_K\neq 0$, we obtain a contradiction with \eqref{hypo2}, as $f_{W_k^1}(0)\notin\mathcal{N}(f)(0)$.
			\item Else, $K:=\max\{j\in\llbracket 0,k'-1\rrbracket; \ (\alpha_{k-k'+j}, \beta_j)\neq (0,0)\}$ is well defined. As $f_0(0)=0$, $f_{B_2}(0)=0$, with 
			\begin{equation*}\begin{gathered}B_2:=[B_10^{k'-1-K},B_10^{k'-K}]\in\alpha_{k-k'+K}^2W_k^1+\beta_K^2W_{k'}^2\\+2(-1)^{\lfloor\frac{k-k'}{2}\rfloor}\alpha_{k-k'+K}\beta_KC_{k+k'-1}+\mathcal{N}+{S}_{\llbracket 3,2\nu\rrbracket}\cap {S}_{\llbracket 0,2\nu\rrbracket,\llbracket 0,2\nu'\rrbracket}(X),\end{gathered}\end{equation*} the expansion is given by Lemma \ref{bracketdev2} with $l=K$.
			Using \eqref{inclusion}, we finally obtain
			$$\alpha_{k-k'+K}^2f_{W_k^1}(0)+\beta_K^2f_{W_{k'}^2}(0)+2(-1)^{\lfloor\frac{k-k'}{2}\rfloor}\alpha_{k-k'+K}\beta_Kf_{C_{k+k'-1}}(0)\in\mathcal{N}(f)(0).$$  We use Lemma \ref{contr} to obtain a contradiction. The first point is obtained with $B=0$.
		\end{enumerate}
	\end{proof}
	\subsection{Closed-loop estimate}
	\begin{lm}\label{cloop2}
		Let $k,k',m,m'\in\nn^*$ be such that $k'\leq k$ and $\nu:=\left\lfloor\frac{\pi}{2}\right\rfloor$, $\nu':=\left\lfloor\frac{\pi'}{2}\right\rfloor$. Assume that \eqref{hypo2} holds. Then, as $(t,\left\|(u,v)\right\|_{L^1})\to0$,
		\begin{equation}\begin{gathered}\label{closed-lopp2}|\left(u_1,\cdots,u_k,v_1,\cdots,v_{k'}\right)(t)|=\mathcal{O}\left(t^{\frac 12}\left\|(u_k,v_{k'})\right\|_{L^2}+\left\|u\right\|_{L^1}^{\nu+1}\right.\\\left.+\left\|v\right\|_{L^1}^{\nu'+1}+\left|x(t;(u,v))\right|\right).\end{gathered}\end{equation}
	\end{lm}
	\begin{proof} We prove the estimate as Lemma \ref{cloop}. For that, we use the asymmetrical Magnus-type representation formula given by Proposition \ref{magnus-assym} with $M=\nu$, $N=\nu'$ instead of Proposition \ref{magnus}. 
	\end{proof}
	\subsection{Proof of the drift}
	\begin{proof}[Proof of Theorem \ref{theoremdriftbis2}] Let $k,k',m,m'\in\nn^*$ and $p,p'\in [1,+\infty]$ be such that  $k'\leq k$. Let $e_1:=f_{W_k^1}(0)$, $e_2:=f_{W_{k'}^2}(0)$ and $e_3:=f_{C_{k+k'-1}}(0)$. Let $\mathbb{P}$ be a linear form given by (BC). The asymmetrical Magnus expansion formula given by Proposition \ref{magnus-assym} with $M=\pi$, $N=\pi'$, the equalities \eqref{egbase-s2bad1}, \eqref{egbase-s2bad2} and \eqref{dvlplog2} give, as $(t,\left\|(u,v)\right\|_{L^1})\to0$,
		\begin{equation*}\begin{gathered}\mathbb{P}x(t;(u,v))=\int_0^t\left(\mathbb{P}(e_1)\frac{u_k^2}{2}+\mathbb{P}(e_2)\frac{v_{k'}^2}{2}\right)+\mathbb{P}(e_3)\xi_{C_{k+k'-1}}(t,(u,v))+\mathcal{O}\left(t\left\|(u_k,v_{k'})\right\|_{L^2}^2\right.\\\left.+|(u_1,\cdots,u_k,v_1,\cdots,v_{k'})(t)|^2+\left\|u\right\|^{\pi+1}_{L^1}+\left\|v\right\|^{\pi'+1}_{L^1}+\left|x(t;(u,v))\right|^{1+\frac{1}{\pi}}\right).\end{gathered}\end{equation*}
		\begin{enumerate}
			\item If $k'\leq k-2$, then, we can apply Lemma \ref{coordinates-ipp} with $j=k+k'-1$, $l=0$ and $N=k-2-\lfloor\frac{k+k'-1}{2}\rfloor$ to obtain
			$$\xi_{C_{k+k'-1}}(t,(u,v))=\sum_{\mu=0}^N(-1)^{\mu}u_{\lfloor\frac{k+k'-1}{2}\rfloor+2+\mu}(t)v_{\lfloor\frac{k+k'}{2}\rfloor-\mu}(t)+(-1)^{N+1}\int_0^tu_kv_{k'}.$$
			\item If $k'\in\{k-1,k\}$, the equation \eqref{egbase-s2good} leads to $\xi_{C_{k+k'-1}}(t,(u,v))=\displaystyle\int_0^tu_kv_{k'}$ and the writing is already convenient. In these cases, $N+1=0$.
		\end{enumerate}
		In all cases, the following equality holds
		\begin{equation}\label{ipp-xi}
			\xi_{C_{k+k'-1}}(t,(u,v))=(-1)^{N+1}\int_0^tu_kv_{k'}+\mathcal{O}\left(|(u_{1},\cdots,u_k)(t)|^2 +t\left\|v_{k'}\right\|_{L^2}^2\right).
		\end{equation}
		The fact that $2(\nu+1)\geq \pi+1$, $2(\nu'+1)\geq \pi'+1$, the closed-loop estimates \eqref{closed-lopp2} and the equality \eqref{ipp-xi}  give, as $\left\|(u,v)\right\|_{L^1}\to 0$,
		\begin{equation*}\begin{gathered}\mathbb{P}x(t;(u,v))=\int_0^t\left(\mathbb{P}(e_1)\frac{u_k^2}{2}+\mathbb{P}(e_2)\frac{v_{k'}^2}{2}+(-1)^{N+1}\mathbb{P}(e_3)u_{k}v_{k'}\right)\\+\mathcal{O}\left(t\left\|(u_k,v_{k'})\right\|_{L^2}^2+\left\|u\right\|^{\pi+1}_{L^1}+\left\|v\right\|^{\pi'+1}_{L^1}+\left|x(t;(u,v))\right|^{1+\frac{1}{\pi}}\right).\end{gathered}\end{equation*}
		Finally, we use the interpolation inequality \eqref{gn-m} to obtain
		\begin{equation*}\begin{gathered}\mathbb{P}x(t;(u,v))=\int_0^t\left(\mathbb{P}(e_1)\frac{u_k^2}{2}+\mathbb{P}(e_2)\frac{v_{k'}^2}{2}+(-1)^{N+1}\mathbb{P}(e_3)u_{k}v_{k'}\right)\\+\mathcal{O}\left(\left(t+t^{\pi-2k}\left\|u\right\|_{W^{m,p}}^{\pi-1}\right)\left\|u_k\right\|_{L^2}^2+\left(t+t^{\pi'-2k'}\left\|v\right\|_{W^{m',p'}}^{\pi'-1}\right)\left\|v_{k'}\right\|_{L^2}^2+\left|x(t;(u,v))\right|^{1+\frac{1}{\pi}}\right).\end{gathered}\end{equation*}
		Let  $\alpha:=\frac{\pi-2k}{\pi-1}$, $\alpha':=\frac{\pi'-2k'}{\pi'-1}$ and $\Delta : (u,v)\in L^1_{\mathrm{loc}}(\rr^+,\rr)^2\mapsto \displaystyle\int_0^t\left( u_k^2+v_{k'}^2\right)\in\rr^+$. We prove that the system \eqref{affine-syst} has a drift along $e_1+e_2$, parallel to $\mathcal{N}(f)(0)$ with strength $\Delta$ in regime  $\left(t,t^{\alpha}\left\|(u,v)\right\|_{W^{m,p}\times W^{m',p'}}\right)\to 0$, as before. This concludes the proof of Theorem \ref{theoremdriftbis2}.
	\end{proof}
	\subsection{Towards a quartic obstruction result}\label{quartic}
	Having dealt with the case of quadratic drifts at any order, a natural question is to ask what happens for the case of quartic drifts and more generally, for any even-order drift. We recall that this situation has been done in the case of single-input control-affine systems by Stefani and Sussmann in \cite{stefani} -- see Theorem \ref{propstef}. Here, we have to study a quartic competition between $5$ brackets, 
	\begin{equation*}\mathfrak{b}_i:=\ad_{X_2}^i\left(\ad_{X_1}^{4-i}(X_0)\right), \quad i\in\llbracket 0,4\rrbracket.\end{equation*}
	The associated functionals are the following ones:
	$\xi_{\mathfrak{b}_i}(t,(u,v))=\displaystyle\frac{1}{i!(4-i)!}\int_0^tu_1^{4-i}v_1^i$, for $i\in\llbracket 0,4\rrbracket.$
	If we choose $\mathcal{N}\subset\Br(X)$, a set of brackets and $\mathbb{P}$, a linear form satisfying $\restriction{\mathbb{P}}{\mathcal{N}(f)(0)}\equiv 0$ so that the dominant term of $\mathbb{P}\mathcal{Z}_{4}(t;f,(u,v))(0)$ are theses terms, then
	\begin{equation}\label{quartic1}\mathbb{P}\mathcal{Z}_4(t;f,(u,v))(0)\simeq\sum_{i=0}^4\frac{\mathbb{P}(e_i)}{i!(4-i)!}\int_0^tu_1^{4-i}v_1^i(s)\ds,\end{equation}
	with $e_i=f_{\mathfrak{b}_i}(0)$ for $i\in\llbracket 0,4\rrbracket$. Using Young's inequality in \eqref{quartic1}, we have
	\begin{equation}\begin{gathered}\mathbb{P}\mathcal{Z}_4(t;f,(u,v))(0)\gtrsim\left(\frac{\mathbb{P}(e_0)}{24}-\frac{\left|\mathbb{P}(e_1)\right|}{8}-\frac{\left|\mathbb{P}(e_3)\right|}{24}\right)\left\|u_1\right\|_{L^4}^4+\frac{\mathbb{P}(e_2)}{4}\int_0^tu_1^2v_1^2\\+\left(\frac{\mathbb{P}(e_4)}{24}-\frac{\left|\mathbb{P}(e_3)\right|}{8}-\frac{\left|\mathbb{P}(e_1)\right|}{24}\right)\left\|v_1\right\|_{L^4}^4,\end{gathered}\end{equation}
	Thus, the condition (BC) can be adapted in this case as: let $e_0,e_1,e_2,e_3,e_4\in\rr^d$ be five vectors and $N\subset\rr^d$ a vector subspace. We say that $e_0,e_1,e_2,e_3,e_4,N$ verify (BC) if there exists $\mathbb{P}:\rr^d\to\rr$ a linear form such that:
	\begin{equation*}3\left|\mathbb{P}(e_1)\right|+\left|\mathbb{P}(e_3)\right| < \mathbb{P}(e_0), \qquad \left|\mathbb{P}(e_1)\right|+3\left|\mathbb{P}(e_3)\right| < \mathbb{P}(e_4), \qquad
		\mathbb{P}(e_2)\geq 0, \qquad \restriction{\mathbb{P}}{N}\equiv 0.\end{equation*}
	With this condition and a good choice of $\mathcal{N}$, we can prove that $(f_1(0),f_2(0))$ is a linearly independent family. Thus, we can use the same strategy for the closed-loop estimates. Moreover, the remainder term of the Magnus-type representation formula, shaped as $\left\|(u,v)\right\|_{L^1}^M$, can be estimated by Gagliardo--Nirenberg interpolation inequalities. The major difficulty lies in extracting the dominant part of $\mathcal{Z}_4(t;f,(u,v))(0)$. This requires the ability to estimate the coordinates of the pseudo-first kind. However, this is very time-consuming: $\mathcal{B}_2$ is made up of $3$ families, $\mathcal{B}_3$ of $8$ families and $\mathcal{B}_4$ is made up of $36$ different families! Thus, this is very tedious to extend Proposition \ref{estimationcoordonnées} in the case of $\mathcal{B}_4$. This strategy would make it possible to treat systems as \eqref{ex-quartic}.
	
	\appendix
	\section{Postponed proofs}
		\subsection{Bracket condition and positive definite quadratic form}
	\begin{lm}
		Let $\alpha,\beta,\gamma\in\rr$ and $q:(a_1,a_1)\in\rr^2\mapsto\frac{1}{2}\alpha a_1^2+\frac{1}{2}\beta a_2^2+\gamma a_1a_2.$ Then, $q$ is a positive definite quadratic form  if and only if $\alpha>0$ and $\gamma^2<\alpha\beta$.
	\end{lm}
	\begin{proof}
		If $q$ is positive definite, then $q(1,0)>0$ so $\alpha>0$. Moreover, for all $a_1\in\rr$, $q(a_1,1)>0$. Thus $\Delta =\gamma^2-\alpha\beta<0$. Conversely, the result follows from the equality: for all $a_1,a_2\in\rr$,
		$$q(a_1,a_2)=\frac{1}{2}\alpha\left(\left(a_1+\frac{\gamma}{\alpha}a_2\right)^2+\frac{\alpha\beta-\gamma^2}{\alpha^2}a_2^2\right).$$
	\end{proof}
	\begin{crl}\label{bcquadratic}
		Let $e_1,e_2,e_3\in\rr^d$ be three vectors and $N\subset\rr^d$ a vector subspace. Let $\sigma:\rr^d\to\rr^d/N$ be the canonical surjection and $q:(a_1,a_2)\in\rr^2\mapsto\frac{1}{2}a_1^2e_1+\frac{1}{2}a_2^2e_2+a_1a_2e_3\in\rr^d.$
		Then, the following are equivalent
		\begin{enumerate}
			\item $e_1,e_2,e_3,N$ satisfy (BC),
			\item there exists a linear form $\mathbb{P}:\rr^d\to\rr$ such that $\restriction{ \mathbb{P}}{N}\equiv 0$ and $(a_1,a_2)\in\rr^2\mapsto \mathbb{P}(q(a_1,a_2))$ is a positive definite quadratic form,
			\item there exists a linear form $\tilde{\mathbb{P}}:\rr^d/N\to\rr$ such that $(a_1,a_2)\in\rr^2\mapsto \tilde{\mathbb{P}}(\sigma\left(q(a_1,a_2)\right))$ is a positive definite quadratic form.
		\end{enumerate}
	\end{crl}
This corollary establishes the connection between Theorem \ref{theoremdriftbis} and the necessary conditions presented in \cite{refId0, lewis} -- see Sections \ref{stateofartrefId0} and \ref{stateofartlewis}.
	\begin{rmq}
		Let $k\in\nn^*$. We fix $\mathfrak{b}_1=W_k^1$,  $\mathfrak{b}_2=W_k^2$ and  $\mathfrak{b}_3=C_{2k-1}$. With the linear form $\mathbb{P}:\rr^d\to\rr$ given by (BC), the quantity $\Delta$ introduced in \eqref{heuristique-p} is a positive definite quadratic form (the expressions of $\xi_{\mathfrak{b}_i}$ are given in \eqref{egbase-s2bad1}, \eqref{egbase-s2bad2}, \eqref{egbase-s2good}).
	\end{rmq}
	\subsection{Geometric conditions}
	\begin{prop}\label{bclink}
		Let $e_1,e_2,e_3\in\rr^d$ be three vectors and $N\subset\rr^d$ a vector subspace.  Let $\sigma:\rr^d\to\rr^d/N$ be the canonical surjection and $\tilde{e}_i:=\sigma(e_i)$ for $i\in\llbracket 1,3\rrbracket$. Then,
		$e_1$, $e_2$, $e_3$, $N$  \textbf{do not} satisfy (BC) if and only if one of the following conditions is satisfied
		\begin{enumerate}
			\item[\textbullet] $\tilde{e}_1=0$ or $\tilde{e}_2=0$,
			\item[\textbullet] $(\tilde{e}_1,\tilde{e}_2)$ is a linearly independent family and $\tilde{e}_3=a\tilde{e}_1+b\tilde{e}_2$ with $ab\geq \frac{1}{4}$,
			\item[\textbullet] $\tilde{e}_2=\beta\tilde{e}_1$ with $\beta<0$,
			\item[\textbullet] $\tilde{e}_2=\beta \tilde{e}_1$, $\tilde{e}_3=\gamma \tilde{e}_1$ with $\beta\leq\gamma^2$ and $\beta\neq 0$.
		\end{enumerate}
	\end{prop}
	\begin{proof}[Proof of Proposition \ref{bclink}]
		Assume that (BC) is not satisfied and that the points $1$, $3$ and $4$ are not verified. The purpose is to show that the second one is. Then, one of the three following possibilities holds
		\begin{enumerate}
			\item[$a.$] $(\tilde{e}_1,\tilde{e}_2)$ is a linearly independent family,
			\item[$b.$] $\tilde{e}_2=\beta\tilde{e}_1$ with $\beta >0$ and $(\tilde{e}_1,\tilde{e}_3)$ is a linearly independent family, 
			\item[$c.$] $\tilde{e}_1\neq 0$, $\tilde{e}_2=\beta\tilde{e}_1$ and $\tilde{e}_3=\gamma\tilde{e}_1$ with $\gamma^2<\beta$.
		\end{enumerate}
		If the point $b.$ holds, then $\spn(e_1)\oplus\spn(e_3)\oplus N$. In this situation, we can define $\mathbb{P}$ as
		$$\mathbb{P}(e_1)=1, \qquad \mathbb{P}(e_3)=0, \qquad \restriction{\mathbb{P}}{N}\equiv0.$$
		Then, $\mathbb{P}$ satisfies (BC). This is a contradiction.	If the point $c.$ holds, then $\spn(e_1)\oplus N$. Thus, we can define $\mathbb{P}$ as
		$$\mathbb{P}(e_1)=1, \qquad \restriction{\mathbb{P}}{N}\equiv0.$$
		Then, $\mathbb{P}(e_3)^2-\mathbb{P}(e_1)\mathbb{P}(e_2)=\gamma^2-\beta<0$ and $\mathbb{P}$ satisfies (BC), this is a contradiction. Necessarily, $a.$ holds, \textit{i.e.}\ $\mathbf{(\tilde{e}_1,\tilde{e}_2)}$ \textbf{is a linearly independent family}. If $\dim\left(\spn(\tilde{e}_1,\tilde{e}_2,\tilde{e}_3)\right)=3$, then $\spn(e_1)\oplus\spn(e_2)\oplus\spn(e_3)\oplus N$. Thus, we can define $\mathbb{P}$ as
		$$\mathbb{P}(e_1)=1, \qquad \mathbb{P}(e_2)=1, \qquad \mathbb{P}(e_3)=0, \qquad \restriction{\mathbb{P}}{N}\equiv0.$$
		Once again, $\mathbb{P}$ satisfies (BC), this is a contradiction. Consequently, \textbf{there exist} $\mathbf{a,b\in\rr}$ \textbf{such that} $\mathbf{\tilde{e}_3=a\tilde{e}_1+b\tilde{e}_2}$. Finally, assume that $ab<\frac 14$. 
		\begin{enumerate}
			\item If $a=0$, then, we can define $\mathbb{P}$ as
			$$\mathbb{P}(e_1)=b^2+1, \qquad \mathbb{P}(e_2)=1,\qquad \restriction{\mathbb{P}}{N}\equiv0.$$
			Then, $\mathbb{P}$ satisfies (BC).
			\item Otherwise, $Q:=a^2x^2+(2ab-1)x+b^2$ verifies $\Delta=1-4ab>0$. Then $x^*:=\frac{1-2ab}{2a^2}>0$ satisfies $Q(x^*)<0$. We can then define $\mathbb{P}$ as
			$$\mathbb{P}(e_1)=x^*, \qquad \mathbb{P}(e_2)=1,\qquad \restriction{\mathbb{P}}{N}\equiv0.$$
			Then, $\mathbb{P}$ satisfies (BC).
		\end{enumerate}
		This is a contradiction. Consequently, $\mathbf{ab\geq\frac 14}$. This is the desired property.
		\bigskip\\
		Conversely, we reason by contraposition. Assume that (BC) holds and let $\mathbb{P}$ be such a linear form.
		\begin{enumerate}
			\item If the point $1$ is satisfied, $e_1\in N$ and $\mathbb{P}(e_1)=0$ or  $e_2\in N$ and $\mathbb{P}(e_2)=0$. 
			\item If the point $2$ holds, $a\neq 0$ and
			$$\mathbb{P}(e_3)^2-\mathbb{P}(e_1)\mathbb{P}(e_2)=a^2\left(\mathbb{P}(e_1)+\left(\frac ba-\frac{1}{2a^2}\right)\mathbb{P}(e_2)\right)^2+\frac{4ab-1}{4a^2}\mathbb{P}(e_2)^2\geq 0.$$
			\item If the point $3$ is satisfied, $\mathbb{P}(e_1)\mathbb{P}(e_2)=\beta\mathbb{P}(e_1)^2\leq 0$.
			\item If the point $4$ holds, $\mathbb{P}(e_3)^2-\mathbb{P}(e_1)\mathbb{P}(e_2)=(\gamma^2-\beta)\mathbb{P}(e_1)^2\geq 0$.
		\end{enumerate}
		All these points are in contradiction with (BC).
	\end{proof}
	\subsection{A bracket expansion in the symmetrical case}\label{appendicebracket}
	\begin{rmq}
		The space $\mathcal{A}(X)$ -- see Definition \ref{def:free-algebra} -- can be seen as a graded algebra
		$\mathcal{A}(X)=\bigoplus_{n\in\nn}\mathcal{A}_n(X),$ where $\mathcal{A}_n(X)$ is the finite-dimensional $\rr$-vector space spanned by monomials of degree $n$ over $X$. The space $\mathcal{L}(X)$ -- see Definition \ref{def:free-lie-algebra} -- is a graded Lie algebra
		$$\mathcal{L}(X)=\bigoplus_{n\in\nn}\mathcal{L}_n(X), \qquad\qquad [\mathcal{L}_n(X),\mathcal{L}_m(X)]\subset\mathcal{L}_{n+m}(X),$$
		where we define $\mathcal{L}_n(X) := \mathcal{L}(X)\cap\mathcal{A}_n(X)$, for any integer $n\in\nn$.
	\end{rmq}
	\begin{proof}
		[Proof of Lemma \ref{bracketdev}]
		Let $l\in\llbracket 0,k-1\rrbracket$, $(\alpha_{j,1})_{j\in\llbracket 0,l\rrbracket}$, $(\alpha_{j,2})_{j\in\llbracket 0,l\rrbracket}\in\rr^{l+1}$. To prove the desired relation, we compute in the quotient space $$\mathcal{L}(X)/\spn\{\ev(b); \  b\in\Br(X), \ n(b)=2, \ n_0(b)<2k-1\}.$$ We note $\bar{a}$ the class of $a\in\mathcal{L}(X)$ in this quotient. Expending the bracket, we have
		$$[B0^{k-l-1},B0^{k-l}]=\sum_{\substack{i,i'\in\{1,2\}\\j,j'\in\llbracket 0,l\rrbracket}}\alpha_{j,i}\alpha_{j',i'}[M_{j+k-l-1}^i,M_{j'+k-l}^{i'}].$$
		We note that, for all $i,i'\in\{1,2\}$,
		\begin{equation*}\forall j,j'\in\llbracket 0,l\rrbracket\text{ such that } j+j'<2l, \quad n_0\left(\left[M_{j+k-l-1}^i,M_{j'+k-l}^{i'}\right]\right)<2k-1.\end{equation*} 
		Using this remark, 
		\begin{equation*}\overline{[B0^{k-l-1},B0^{k-l}]}=\sum_{i,i'\in\{1,2\}}\alpha_{l,i}\alpha_{l,i'}\overline{[M_{k-1}^i,M_{k}^{i'}]}.
		\end{equation*}Finally, using Jacobi's inequality \eqref{jacobibase},
		\begin{equation*}\overline{[B0^{k-l-1},B0^{k-l}]}=\alpha_{l,1}^2\overline{W_k^1}+\alpha_{l,2}^2\overline{W_k^2}+\alpha_{l,1}\alpha_{l,2}\left(2\overline{C_{2k-1}}+\overline{C_{2k-2,1}}\right).\end{equation*}
		As $\mathcal{L}(X)$ is a graded Lie algebra, $\mathcal{B}_{2,2k-2}:=\{\ev(b); \ b\in\Br(X), \ n(b)=2,\ n_0(b)<2k-1\}$ generates all the elements $\ev(b)$ with $n(b)=2$ and $n_0(b)<2k-1$. The elements of $\mathcal{B}_{2,2k-2}$ are in $\mathcal{N}_k^m$. Finally, as $C_{2k-2,1}\in\mathcal{N}_k^m$, the desired result follows.
	\end{proof}
	\subsection{A bracket expansion in the asymmetrical case}\label{appendicebracket2}
	The purpose of this subsection is to prove the expansion of Lemma \ref{bracketdev2}.
	The proof of this lemma is quite different from the case $k=k'$ studied in Lemma \ref{bracketdev} and is based on the following lemma.
	\begin{lm}\label{expansion} The following expansions hold.
		\begin{enumerate}
			\item For any $\nu\in\nn$ and $a,b\in\mathcal{L}(X)$, 
			\begin{equation}\label{jacobi}
				[a,b0^{\nu}]=\sum_{\nu'=0}^{\nu}\binom{\nu}{\nu'}(-1)^{\nu'}[a0^{\nu'},b]0^{\nu-\nu'}.
			\end{equation}
			\item For any $\nu\in\nn^*$, there exist coefficients $\alpha_r^{\nu}\in\zz$ for $1\leq 2r+1\leq \nu$ such that, for every $b\in\mathcal{L}(X)$,
			\begin{equation}\label{decompuu}
				[b,b0^{\nu}]=\sum_{1\leq 2r+1\leq\nu}\alpha_r^{\nu}[b0^r,b0^{r+1}]0^{\nu-2r-1}.
			\end{equation}
			\item For any $\nu\in\nn$, there exist coefficients $\beta_r^{\nu}\in\zz$ for $0\leq r\leq \nu$ such that, for every $p\in\nn$,
			\begin{equation}\label{decompuv}
				[M_p^1,M_p^20^{\nu}]=\sum_{r=0}^{\nu}\beta_r^{\nu}C_{2p+r,\nu-r}.
			\end{equation}
			\item For any $\nu\in\nn$, there exist coefficients $\gamma^{\nu}_r\in\zz$ for $0\leq r\leq \nu$ such that, for every $p\in\nn$,
			\begin{equation}\label{decompvu}
				[M_p^2,M_p^10^{\nu}]=\sum_{r=0}^{\nu}\gamma_r^{\nu}C_{2p+r,\nu-r}.
			\end{equation}
			Moreover, $\gamma^{\nu}_{\nu}=(-1)^{1+\lfloor\frac{\nu+1}{2}\rfloor}$.
		\end{enumerate}
	\end{lm}
	\begin{proof}
		The first two points are proved in \cite[Lemma $4.11$]{beauchard2024unified}. We prove the last point by induction on $\nu$ (the proof of $3.$ is very similar): the equality is true for $\nu=0$ with $\gamma_{0}^0=-1$. 
		The equality if true for $\nu=1$ with $\gamma_0^1=0$ and $\gamma_1^1=1$.
		We assume that the formula holds for $\nu,\nu+1$, with $\nu\geq 0$. Then, for every $p\in\nn$, 
		$$[M_p^2,M_p^10^{\nu+2}]=[M_p^2,M_p^10^{\nu+1}]0-[M_{p+1}^2,M_{p+1}^10^{\nu}],$$ thanks to the Jacobi's equality \eqref{jacobibase}.
		Using the induction hypothesis and a change of variable, we get
		$$[M_p^2,M_p^10^{\nu+2}]=\sum_{r=0}^{\nu+1}\gamma_r^{\nu+1}C_{2p+r,(\nu+2)-r}-\sum_{r=2}^{\nu+2} \gamma_{r-2}^{\nu}C_{2p+r,(\nu+2)-r}.$$ 
		Thus,
		$[M_p^2,M_p^10^{\nu+2}]=\displaystyle\sum_{r=0}^{\nu+2}\gamma_r^{\nu+2}C_{2p+r,\nu+2-r},$
		with 
		$$\forall r\in\llbracket 2, \nu+1\rrbracket, \quad \gamma_r^{\nu+2}=\gamma_r^{\nu+1}-\gamma_{r-2}^{\nu}, \qquad \gamma_{\nu+2}^{\nu+2}=-\gamma_{\nu}^{\nu}, \qquad \gamma_0^{\nu+2}=\gamma_0^{\nu+1}, \qquad \gamma_1^{\nu+2}=\gamma_1^{\nu+1}.$$ 
		We obtain the desired equality, as $\gamma_{\nu+2}^{\nu+2}=-\gamma_{\nu}^{\nu}=-(-1)^{1+\lfloor\frac{\nu+1}{2}\rfloor}=(-1)^{1+\lfloor\frac{(\nu+2)+1}{2}\rfloor}.$
	\end{proof}
	We are now in a position to prove Lemma \ref{bracketdev2}.
	\begin{proof}[Proof of Lemma \ref{bracketdev2}] By definition, \begin{equation}\label{decomp}\left[B0^{k'-l-1},B0^{k'-l}\right]=(\text{I})+(\text{II})+(\text{III})+(\text{IV}),
		\end{equation}
		with $$\begin{array}{cc}\text{(I)}=\displaystyle\sum_{i,j=0}^{l+k-k'}\alpha_i\alpha_j[M_{i+k'-l-1}^1,M_{j+k'-l}^1], & \text{(II)}=\displaystyle\sum_{i=0}^{l+k-k'}\sum_{j=0}^{l}\alpha_i\beta_j[M_{i+k'-l-1}^1,M_{j+k'-l}^2],\\
			\text{(III)}=\displaystyle\sum_{i=0}^l\sum_{j=0}^{l+k-k'}\beta_i\alpha_j[M_{i+k'-l-1}^2,M_{j+k'-l}^1],& \text{(IV)}=\displaystyle\sum_{i,j=0}^l\beta_i\beta_j[M_{i+k'-l-1}^2,M_{j+k'-l}^2].\end{array}$$
		Then, 
		\begin{equation}\label{decp}(\text{I})=\sum_{i=k'-l}^{k}\alpha_{i+l-k'}^2W^1_i+\left(\sum_{i=2}^{l+k-k'}\sum_{j=0}^{i-2}+\sum_{i=0}^{l+k-k'-1}\sum_{j=i+1}^{l+k-k'}\right)\alpha_i\alpha_j[M_{i+k'-l-1}^1,M_{j+k'-l}^1],
		\end{equation}
		as the bracket is zero if $j=i-1$. Moreover, the equation \eqref{decompuu}, applied with $b=M_{j+k'-l}^1$ and $\nu=i-j-1\geq 1$ gives: for all $2\leq i\leq l+k-k'$, $ 0\leq j\leq i-2$,
		$$[M_{i+k'-l-1}^1,M_{j+k'-l}^1]=-\sum_{1\leq 2r+1\leq i-j-1}\alpha_r^{i-j-1}W^1_{j+k'-l+r+1,i-j-2r-2}.$$
		As $j+k'-l+r+1\leq k-1$, we obtain
		\begin{equation}\label{dec2p}\text{for all }2\leq i\leq l+k-k', \  0\leq j\leq i-2,\quad [M_{i+k'-l-1}^1,M_{j+k'-l}^1]\in\mathcal{N}_{k,k'}^{m,m'}.\end{equation}
		Similarly, we obtain:
		\begin{equation}\label{dec3p}\text{for all } 0\leq i\leq l+k-k'-1, \  i+1\leq j\leq l+k-k',\quad [M_{i+k'-l-1}^1,M_{j+k'-l}^1]\in\mathcal{N}_{k,k'}^{m,m'}.\end{equation}
		Thus, the equations \eqref{decp}, \eqref{dec2p} and \eqref{dec3p} give
		\begin{equation}\label{dec4p}\text{(I)}-\alpha_{l+k-k'}^2W_k^1\in\mathcal{N}_{k,k'}^{m,m'}.\end{equation}
		We can manipulate the term (IV) in the same way to obtain
		\begin{equation}\label{dec5p}\text{(IV)}-\beta_l^2W_{k'}^2\in\mathcal{N}_{k,k'}^{m,m'}.\end{equation}
		Finally, we need to examine the cross terms (II) and (III).
		\begin{equation}\label{dec6p}(\text{II})=\alpha_{l+k-k'}\beta_l[M^1_{k-1},M^2_{k'}]+\sum_{\substack{(i,j)\in\llbracket 0,l+k-k'\rrbracket\times\llbracket 0,l\rrbracket\\(i,j)\neq (l+k-k',l)}}\alpha_i\beta_j[M_{i+k'-l-1}^1,M_{j+k'-l}^2].
		\end{equation}
		Assume temporarily that $k\neq k'$. The equation \eqref{decompvu}, applied with $p=k'$ and $\nu=k-k'-1\geq 0$ gives
		$$[M_{k-1}^1,M_{k'}^2] =(-1)^{\lfloor\frac{k-k'}{2}\rfloor}C_{k+k'-1}-\sum_{r=0}^{k-k'-2}\gamma_r^{k-k'-1}C_{2k'+r,k-k'-1-r}.$$
		As, in the sum, $2k'+r\leq k+k'-2$, we obtain
		\begin{equation}\label{dec7p}[M_{k-1}^1,M_{k'}^2] -(-1)^{\lfloor\frac{k-k'}{2}\rfloor}C_{k+k'-1}\in\mathcal{N}_{k,k'}^{m,m'}.\end{equation}
		Using the Jacobi's formula \eqref{jacobibase}, this equality is also true when $k=k'$. We expand on the basis the second term of the right-hand side of \eqref{dec6p}. We split the space of subscripts as
		$$\llbracket 0,l+k-k'\rrbracket\times\llbracket 0,l\rrbracket\setminus\{(l+k-k',l)\}=A\sqcup B\sqcup C\sqcup D\sqcup E\sqcup F,$$
		with
		$$A=\{(i,j); \ 1\leq i\leq l-1, \ 0\leq j\leq i-1\}, \quad B=\{(i,j); \ 0\leq i\leq l-1,\  i\leq j\leq l-1\},$$
		$$C=\{(i,j); \ l\leq i\leq l+k-k', \ 0\leq j\leq l-1\}, \quad D=\llbracket 0,l-1\rrbracket\times\{l\},$$$$ E=\llbracket l+1,l+k-k'-1\rrbracket\times\{l\}, \quad F=\{(l,l)\}.$$
		Note that the spaces $E$ and $F$ are empty if $k=k'$. For all $(i,j)\in A$, we can apply \eqref{decompvu} with $p=j+k'-l$ and $\nu=i-j-1$ to have
		$$[M^1_{i+k'-l-1},M^2_{j+k'-l}]=-\sum_{r=0}^{i-j-1}\gamma_r^{i-j-1}C_{2(j+k'-l)+r,i-j-1-r}.$$
		As $2(j+k'-l)+r\leq k+k'-4$, we have
		$$\forall (i,j)\in A, \quad [M^1_{i+k'-l-1},M^2_{j+k'-l}]\in\mathcal{N}_{k,k'}^{m,m'}.$$
		We can do the same manipulations for $(i,j)\in B$, $C$, $D$, $E$ and $F$, thanks to the equations \eqref{decompuv} and \eqref{decompvu}. Using, \eqref{dec7p} in \eqref{dec6p}, we finally get
		\begin{equation}\label{dec9p}\text{(II)}-(-1)^{\lfloor\frac{k-k'}{2}\rfloor}\alpha_{l+k-k'}\beta_lC_{k+k'-1}\in\mathcal{N}_{k,k'}^{m,m'}.\end{equation}
		With the same manipulations, we have
		\begin{equation}\label{dec10p}\text{(III)}-(-1)^{\lfloor\frac{k-k'}{2}\rfloor}\alpha_{l+k-k'}\beta_lC_{k+k'-1}\in\mathcal{N}_{k,k'}^{m,m'}.\end{equation}
		The equations \eqref{decomp}, \eqref{dec4p}, \eqref{dec5p},  \eqref{dec9p},  and \eqref{dec10p} lead to the conclusion.
	\end{proof}
	\subsection{Relation between Sussmann's $\mathcal{S}(\theta)$-condition and Theorem \ref{theoremdriftbis}}\label{append-linkstheta}
	Let $k\in\nn^*$, $m=1$. We assume that 
	$$f_{W_k^1}(0), \quad f_{W_k^2}(0), \quad f_{C_{2k-1}}(0), \quad \mathcal{N}_k^1(f)(0)\quad \text{satisfy (BC)}.$$
	Let us show that, for every $\theta\in[0,1]$, \eqref{Hyp_Sussm} is not verified for $\mathfrak{b}:=W_k^1$. We assume the opposite: there exists $\theta\in[0,1]$ such that \eqref{Hyp_Sussm} holds for $\mathfrak{b}=W_k^1$. Then, $n_0(\mathfrak{b})=2k-1$ is odd, $n_1(\mathfrak{b})=2$, $n_2(\mathfrak{b})=0$ are even and
	$f_{\sigma(\mathfrak{b})}(0)=f_{W_k^1}(0)+f_{W_k^2}(0).$
	Let $b\in\Br(X)$ be such $n(b)+\theta n_0(b)<n(\mathfrak{b})+\theta n_0(\mathfrak{b})=2+(2k-1)\theta$. Then, 
	$$n(b)< 2+(2k-1)\theta\leq 2k+1.$$
	Moreover, if $n(b)=2$, then
	$$2+\theta n_0(b)<2+\theta(2k-1)\qquad \text{ so }\qquad n_0(b)<2k-1.$$
	Let $\mathcal{E}:={S}_{\llbracket 1,\pi(k,1)\rrbracket\setminus\{2\}}(X)\cup\{b\in\Br(X); \ n(b)=2, \ n_0(b)<2k-1\}$.
	As $\pi(k,1)=2k+1$, the previous inequalities lead to 
	$f_{\sigma(\mathfrak{b})}(0)\in\mathcal{E}(f)(0)$. As $\mathcal{E}(f)(0)\subset\mathcal{N}_k^1(f)(0)$, we obtain
	$$\mathbb{P}\left(f_{W_k^1}(0)\right)+\mathbb{P}\left(f_{W_k^2}(0)\right)=0,$$
	where $\mathbb{P}$ is a linear form given by (BC). Thus, 
	$$0\leq \mathbb{P}\left(f_{C_{2k-1}}(0)\right)^2<\mathbb{P}\left(f_{W_k^1}(0)\right)\mathbb{P}\left(f_{W_k^2}(0)\right)=-\mathbb{P}\left(f_{W_k^1}(0)\right)^2\leq 0.$$
	This is a contradiction. Consequently, \eqref{Hyp_Sussm} is not verified for $\mathfrak{b}=W_k^1$.
	\section*{Acknowledgement}
	
	The author would like to take particular care in thanking Karine Beauchard and Frédéric Marbach for the many discussions that brought this article to life.
	
	The author acknowledges support from grants ANR-20-CE40-0009 and ANR-11-LABX-0020, as well as from the Fondation Simone et Cino Del Duca – Institut de France.
\printbibliography
\end{document}